\documentclass[a4paper,10pt,margin=0.5in]{extarticle}
\usepackage{geometry}
\usepackage[english]{babel}
\usepackage[utf8]{inputenc}
\usepackage{amsmath}
\usepackage{float}
\usepackage{graphicx,subfigure}
\usepackage[utf8]{inputenc}
\usepackage[english]{babel}
\usepackage{amsfonts}
\usepackage{comment} 
\usepackage{tikz}
\usepackage{mathtools}
\usepackage{fixltx2e}
\usepackage{pifont}
\usepackage{caption}
\usepackage[colorinlistoftodos]{todonotes}

\usepackage{amsmath,amsfonts,amssymb,amsthm,epsfig,epstopdf,array}
\newtheorem*{remark}{Remark}
\newtheorem{theorem}{Theorem}
\newtheorem{lemma}[theorem]{Lemma}
\newtheorem{assumption}{Assumption}
\newtheorem{mydef}{Definition}
\usepackage{caption}

\newcommand{\norm}[1]{\left\lVert#1\right\rVert}
\usepackage{verbatim}
\usepackage{balance}
\let\labelindent\relax
\usepackage{enumitem}
\usepackage[bottom]{footmisc}
\usepackage[utf8]{inputenc}
\usepackage[english]{babel}

\usepackage[linkcolor=blue,colorlinks=true]{hyperref}
 \usepackage{cleveref}
\author{Jaskaran Grover$^{1}$, Changliu Liu$^{1}$, Katia Sycara $^{1}$
\thanks{$^{1}$J. Grover, C. Liu, K. Sycara are with the Robotics Institute at
		Carnegie Mellon University, 5000 Forbes Avenue, Pittsburgh, PA 15213, USA.
		{\tt\small \{jaskarag,cliu6,sycara\}@andrew.cmu.edu}}}%
\begin{document}
	\title{\LARGE \bf
		Parameter Identification for Multirobot Systems Using Optimization-based Controllers }
	\date{}
	\maketitle
\begin{abstract}
This paper considers the problem of parameter identification for a multirobot system. We wish to understand when is it feasible for an adversarial observer to reverse-engineer the parameters of tasks being performed by a team of robots by simply observing their positions. We address this question by using the concept of \textit{persistency of excitation} from system identification. Each robot in the team uses optimization-based controllers for mediating between task satisfaction and collision avoidance. These controllers exhibit an \textit{implicit} dependence on the task's parameters which poses a hurdle for deriving necessary conditions for parameter identification, since such conditions usually require an explicit relation. We address this bottleneck by using duality theory and SVD of active collision avoidance constraints and derive an explicit relation between each robot's task parameters and its control inputs. This allows us to derive the main necessary conditions for successful identification which agree with our intuition. We demonstrate the importance of these conditions through numerical simulations by using (a) an adaptive observer and (b) an unscented Kalman filter for goal estimation in various geometric settings. These simulations show that under circumstances where parameter inference is supposed to be infeasible per our conditions, both these estimators fail and likewise when it is feasible, both converge to the true parameters. Videos of these results are available at \url{https://bit.ly/3kQYj5J}
\end{abstract}
\section{Introduction}
\balance
There has been significant research on task-based motion planning and control synthesis for multiple robots, for applications involving search and rescue \cite{kantor2003distributed}, sensor coverage \cite{cortes2004coverage} and environmental exploration \cite{burgard2005coordinated}. Global behaviors result from executing local controllers on individual robots interacting with their neighbors \cite{ogren2002control}, \cite{olfati2007consensus}.  The inverse problem for task based control is task inference  \cite{byravan2015graph}, which is the subject of this paper. 
Our focus is on multirobot task identification by an adversary observing robots performing some tasks. Specifically, we want to understand how easy it is  for the observer to infer parameters of the tasks being performed by the robots, using just their positions.  Such inference can provide an avenue to the adversary to impede task execution.

Consequently, from the perspective of robots, the question is how can they coordinate their motions so as to obfuscate their parameters. On the other hand, for the observer, the question is when is this model-fitting problem well-posed \textit{i.e.} when are the positions of the robots ``rich enough" so that they suffice to reveal parameters of their underlying tasks.  In this paper, we take the view of the observer monitoring a multirobot system in which each robot is tasked with reaching a goal position while avoiding collisions with other robots. The robots use optimization-based controllers which have demonstrated great success for the goal stabilization and avoidance control problem that we are considering \cite{wang2017safety,grover2019deadlock}. \textit{The observer's problem is to estimate goal locations and controller gains of all robots using their positions over some finite time}.  We use goal and gain estimation as an example to illustrate the mechanics of our approach; it can be easily adapted to identify parameters of a different task as well.

The development of provably convergent parameter estimation algorithms for nonlinear systems  has been studied extensively \cite{na2015robust},\cite{adetola2010performance},\cite{hartman2012robust}. These algorithms have been used for identifying parameters of manipulators \cite{yang2018adaptive} and quadrotors \cite{zhao2018online}. We consider these robots ``monolithic" since they do not interact with other agents during the parameter identification phase.  Moreover, these algorithms require explicit relations between plant parameters and dynamics to derive conditions under which these estimators converge.  While deriving such identifiability conditions is our incentive as well, the plant that we focus on has some unique features. It is composed of (a) mutually interacting robots; and (b) each robot uses optimization-based controllers. Interactions among robots pose a challenge to identification because at any given time, the dynamics of an ego robot are governed by the set of robots currently neighboring the ego which is time-varying.  Secondly, having an optimization generate the control input for each robot causes the robot's dynamics to depend \textit{implicitly} on task parameters through the optimization's objective, which prohibits straightforward application of the identifiability conditions derived in these works.
 
To address these challenges, we first provide some background on parameter identification in section \ref{AdaptiveObserverAlgorithm}. Using \textit{persistency of excitation} \cite{adetola2010performance,ioannou2012robust}, we derive a new necessary condition for successful identification in lemma \ref{lemma1}.  In section \ref{ControlReview}, we review multirobot avoidance control using optimization-based controllers and formalize the identification problem for this system in section \ref{ProblemFormulation}.  The main contributions begin from section \ref{KKT Conditions} where we derive the KKT conditions of the control-synthesis optimization. By focusing on the set of active interactions (\textit{i.e.} active constraints) of a robot with other robots, we pose an equality-constrained optimization (EQP) which is the first step for deriving a relation between parameters and the control.  In section \ref{SVD}, we classify each robot's dynamics based on the number of constraints in this EQP, and linear independence relations amongst these constraints. Taking the SVD of these constraints allows us to derive the explicit relation between the parameters and dynamics of each robot that we wanted. Finally, using these relations in conjunction with the \textit{persistency of excitation} requirement and the result derived in lemma \ref{lemma1}, we provide the main neceessary conditions for successful task identification of the multirobot system (theorems (\ref{theorem_caseA}-\ref{theorem_caseD})).  The message that these theorems convey is that as the number of robots that an ego robot interacts with increases, estimation of ego's parameters becomes difficult. This confirms our intuition, because with more interactions, the ego robot's motion (that the observer measures) is expended in satisfying collision avoidance constraints which it achieves by sacrificing task performance.  We demonstrate this numerically in section \ref{Results}, where we use an adaptive observer and a UKF to infer robots' goals using their positions under various geometric settings.   We conclude in \ref{Conclusions} by summarizing and provide directions for future work.

\section{Observer based Parameter Identification}
\label{AdaptiveObserverAlgorithm}
While there exist several parameter estimation algorithms that the observer can potentially leverage (such as RLS \cite{edgar2010recursive}, UKF \cite{thrun2002probabilistic}), we focus on observer based methods borrowing ideas from \cite{adetola2008finite,adetola2010performance,na2015robust}  because they provide conditions under which exponential convergence (and even finite-time convergence) to the true parameters is achieved.
Consider a nonlinear parameter-affine system as follows
\begin{align}
\label{original}
    \dot{\boldsymbol{x}}=G(\boldsymbol{x})\boldsymbol{\theta} + f(\boldsymbol{x}),
\end{align}
where $\boldsymbol{x} \in \mathbb{R}^n$ is the measurable state, $\boldsymbol{\theta}\in\mathbb{R}^p$ is the unknown parameter and $G(\boldsymbol{x}):\mathbb{R}^{n} \longrightarrow \mathbb{R}^{n \times p}$, $f(\boldsymbol{x}):\mathbb{R}^{n} \longrightarrow\mathbb{R}^{n}$ are known functions. In our context, $\boldsymbol{x}(t)$ will correspond to the position of the ego robot under observation and $\boldsymbol{\theta}$ denotes the task parameters of its controller that we wish to infer. We assume that the observer runs several parallel estimators synchronously, one for estimating the parameters of each robot, so the focus here is on the ego robot. The observer's problem is to design an estimation law $\dot{\hat{\boldsymbol{\theta}}}=\psi(\hat{\boldsymbol{\theta}},\boldsymbol{x})$ that guarantees convergence of $\hat{\boldsymbol{\theta}} \longrightarrow \boldsymbol{\theta}$ by using $\boldsymbol{x}(t)$ over some $t \in[0,T]$ where $T$ is large enough. Consider a state predictor defined analogously to \eqref{original}
\begin{align}
\label{predictor}
    \dot{\hat{\boldsymbol{x}}}=G(\boldsymbol{x})\boldsymbol{\theta}^0 + f(\boldsymbol{x}) + k_w(\boldsymbol{x}-\hat{\boldsymbol{x}}), \mbox{    }\hat{\boldsymbol{x}}(0) = \boldsymbol{x}(\boldsymbol{0}),
\end{align}
where $\boldsymbol{\theta}^0 \in \mathbb{R}^p$ is a nominal initial estimate of $\boldsymbol{\theta}$ and $k_w>0$. Define an auxillary variable $\boldsymbol{\eta} \in \mathbb{R}^n$ as follows
\begin{align}
\label{eta}
    \boldsymbol{\eta} = \boldsymbol{x}-\hat{\boldsymbol{x}}-W(\boldsymbol{\theta}-\boldsymbol{\theta}^0)
\end{align}
where $W \in \mathbb{R}^{n \times p}$ is generated according to
\begin{align}
\label{W}
    \dot{W} = -k_wW + G(\boldsymbol{x}), \mbox{    } W(0)=0.
\end{align}
Here $W$ is a low-pass filtered version of $G(\boldsymbol{x})$. While $\boldsymbol{\eta}$ as defined in \eqref{eta} is not measurable because it depends on $\boldsymbol{\theta}$ that is unknown, defining $W$ as in \eqref{W}  lets us generate $\boldsymbol{\eta}$ using
\begin{align}
\label{etadot}
\dot{\boldsymbol{\eta}} = -k_w\boldsymbol{\eta}, \mbox{    } \boldsymbol{\eta}(0)=\boldsymbol{x}(0)-\hat{\boldsymbol{x}}(0).
\end{align}
Based on \eqref{original}-\eqref{etadot}, let $Q \in \mathbb{R}^{p \times p}$ and $C\in \mathbb{R}^{p}$ be generated according to the following dynamics
\begin{align}
\label{QCdef}
    \dot{Q} &= W^TW,\mbox{    } Q(0) = 0^{p \times p},  \\
    \dot{C} &= W^T(W\boldsymbol{\theta^0}+\boldsymbol{x}-\hat{\boldsymbol{x}}-\boldsymbol{\eta}), C(0) = 0^{p \times 1},
\end{align}
and let $t_c$ be the time at which $Q(t_c) \succ 0$, then the following parameter update law 
\begin{align}
\label{parameter_update_law}
    \dot{\hat{\boldsymbol{\theta}}}=\Gamma(C - Q\hat{\boldsymbol{\theta}}), \mbox{    } \hat{\boldsymbol{\theta}}(0)=\boldsymbol{\theta}^0,
\end{align}
for $\Gamma \succ 0$ guarantees that $\norm{\hat{\boldsymbol{\theta}}-\boldsymbol{\theta}}$ is non-increasing for $0\leq t \leq t_c$ and exponentially converges to 0 for $t>t_c$.
Thus, as long as there exists  $t_c$ at which $Q(t_c) \succ 0$, convergence of the estimate $\hat{\boldsymbol{\theta}}$ to the true parameter $\boldsymbol{\theta}$ is guaranteed.
\begin{mydef}[Persistency of Excitation \cite{willems2005note}]
\label{PoE}
A bounded, locally square integrable function $\Phi:\mathbb{R}^+ \longrightarrow \mathbb{R}^{n}$ is said to be persistently exciting (\textbf{PE}) if there exist constants $T > 0,\epsilon>0$ such that $\int_{t}^{t+T}\Phi(s)\Phi^T(s)ds \succcurlyeq \epsilon I$ $\forall t \geq 0$
\end{mydef}
\begin{mydef}[Interval Excitation \cite{yang2018adaptive}] 
\label{IE}
A bounded, locally square integrable function $\Phi:\mathbb{R}^+ \longrightarrow \mathbb{R}^{n}$ is said to be interval exciting (\textbf{IE}) if there exist constants $T_0 > 0,\epsilon>0$ such that $\int_{0}^{T_0}\Phi(s)\Phi^T(s)ds \succcurlyeq \epsilon I$
\end{mydef}

\begin{remark}
\label{R2}
The condition that $Q(t_c)\succ 0$ is equivalent to the \textbf{IE} condition presented in Def. \eqref{IE} for $T_0\coloneqq t_c$. Indeed by defining $\Phi(t)\coloneqq W^T(\boldsymbol{x}(t))$, we get that $W^T(\boldsymbol{x}(t))$ is \textbf{IE} iff 
$\int_{0}^{t_c}W^T(\boldsymbol{x}(s))W(\boldsymbol{x}(s))ds \succ 0 \iff Q(t_c) \succ 0$ since $Q(t)= \int_{0}^{t}W^T(\boldsymbol{x}(s))W(\boldsymbol{x}(s))ds$ using \eqref{QCdef}
\end{remark}
\hspace{-0.37cm}Since $W(\boldsymbol{x}(t))$ is a low-pass filtered $G(\boldsymbol{x}(t))$ \eqref{W}, $W^T(\boldsymbol{x}(t))$ is \textbf{IE} only when $G^T(\boldsymbol{x}(t))$ is \textbf{IE} \cite{narendra2012stable}. Therefore, $G^T(\boldsymbol{x}(t)) $ is \textbf{IE} implies  existence of $t_c$ such that $Q(t_c)\succ 0$ \textit{i.e.}
\begin{align}
\label{GisPE}
\int_{0}^{t_c}G^T(\boldsymbol{x}(s))G(\boldsymbol{x}(s))ds \succcurlyeq \epsilon I 
\implies  \exists t_c \mid Q(t_c)\succ 0.
\end{align}
Next, we derive a new necessary condition which is required for $G^T(\boldsymbol{x}(t)) $ to be \textbf{IE}.
\begin{lemma}
\label{lemma1}
Let $\mathcal{N}(G(\boldsymbol{x})) \subset \mathbb{R}^{p}$ denote the null space of $G(\boldsymbol{x})$, then a necessary condition for $G^T(\boldsymbol{x}(t))$ to be \textbf{IE}, is that $\mathcal{N}(G(\boldsymbol{x}))$ must \textbf{not} be time-invariant.
\end{lemma} 
\begin{proof}
We prove this lemma by contradiction. Let $\boldsymbol{v}(\boldsymbol{x}(t)) \in \mathcal{N}(G(\boldsymbol{x}(t)))$ and assume that $\boldsymbol{v}(\boldsymbol{x}(t))$ is time-invariant \textit{i.e.} $\boldsymbol{v}(\boldsymbol{x}(t))\equiv \textbf{v} \in \mathbb{R}^p$ for some constant non-zero vector $\textbf{v}$. Let $T>0$, then we have that
\begin{align}
r&=\textbf{v}^T\bigg(\int_{0}^{T}G^T(\boldsymbol{x}(s))G(\boldsymbol{x}(s))ds\bigg)\textbf{v} \nonumber \\
    &=\int_{0}^{T}\bigg(\textbf{v}^TG^T(\boldsymbol{x}(s))G(\boldsymbol{x}(s))\textbf{v}\bigg)ds \nonumber \\
    &=0 \mbox{ } \forall t>0.
\end{align}
Since we assumed that $\textbf{v}\neq \boldsymbol{0}$ and $T$ was arbitrary,
\begin{align}
r=0 &\implies \int_{0}^{T}G^T(\boldsymbol{x}(s))G(\boldsymbol{x}(s))ds \nsucceq \epsilon I  \mbox{ } \forall t,\epsilon>0 \nonumber \\
&\implies G(\boldsymbol{x}(t)) \mbox{ is  not \textbf{IE} } \nonumber.
\end{align}
Since existence of such a $\textbf{v}$ implies $G^T(\boldsymbol{x}(t))$ is not \textbf{IE}, therefore, $\nexists \mbox{ } t_c$ for which $Q(t_c) \succ 0$.
\end{proof}
\hspace{-0.38cm}Consequently, failure to obtain positive-definiteness of $Q(t)$ prevents unique identification of  $\boldsymbol{\theta}$ using \eqref{parameter_update_law}. What is the intuition for this result? Recall from \eqref{original} that the dynamics depend on the true parameter $\boldsymbol{\theta}$ affinely through $G(\boldsymbol{x})$. A time-invariant vector $\textbf{v} \in \mathcal{N}(G(\boldsymbol{x}(t)))$ qualitatively represents a pathological parameter that does not influence the dynamics because $G(\boldsymbol{x}(t))\textbf{v}=\boldsymbol{0}$ and by extension, also does not influence the measurements $\boldsymbol{x}(t)$. Said another way, suppose $\boldsymbol{\theta}$ is the true parameter of the system and let $\boldsymbol{\theta} + \alpha \textbf{v}$  denote an arbitrary parameter for some $\alpha \in \mathbb{R}$. Then, the following calculation shows that either of these parameters result in the same observed dynamics $\dot{\boldsymbol{x}}(t)$ for any $ t$  for a given initial condition $\boldsymbol{x}(0)$, because $G(\boldsymbol{x}(t))\textbf{v}=\boldsymbol{0}$:
\begin{align}
\label{original2}
    \dot{\boldsymbol{x}}&=G(\boldsymbol{x})(\boldsymbol{\theta}+\alpha\textbf{v}) + f(\boldsymbol{x})\nonumber \\
    &=G(\boldsymbol{x})\boldsymbol{\theta}+\alpha 
    G(\boldsymbol{x})\textbf{v}
    + f(\boldsymbol{x})\nonumber  \\
    &=G(\boldsymbol{x})\boldsymbol{\theta} + f(\boldsymbol{x})\nonumber 
\end{align}
As a result, the observed measurements $\boldsymbol{x}(t)$ would be identical for either choice of parameters (\textit{i.e.} $\boldsymbol{\theta}$ or $\boldsymbol{\theta} + \alpha \textbf{v}$). Therefore, unique identification of the true $\boldsymbol{\theta}$ solely based on these measurements is not possible. 
 
For our application involving parameter estimation for robots, the high-level task and the resulting dynamics of each robot govern the specific form of this condition. In this paper, the task for each robot is to navigate to a goal location while avoiding collisions with the other robots in the system. We use this task as an example to show the mechanics of our inference approach which can be easily adapted to infer parameters of a different type of task as well. We assume that each robot uses optimization-based controllers to achieve this. Most optimization-based controllers ultimately use a Quadratic Program (QP) to compute the control \cite{wei2019safe}. As an example, our framwork uses Control Barrier Function based QPs, just to demonstrate the approach, but it generalizes to any other QP-based controller as well.
\section{Avoidance Control for Multirobot Systems}
\label{ControlReview}
We refer the reader to \cite{wang2017safety} for a formulation for CBF-QP for the multirobot goal-stabilization and collision avoidance problem. We present a brief summary here. Let there be a total of $M+1$ robots in the system. From the perspective of an ego robot, the remaining $M$ robots are ``cooperative obstacles" who share the responsibility of avoiding collisions with the ego robot, while navigating to their own goals.  In the following, the focus is on the ego robot. This robot follows single-integrator dynamics \textit{i.e.}
\begin{align}
\dot{\boldsymbol{x}} = \boldsymbol{u},
\end{align}
where $\boldsymbol{x}=(p_x,p_y)\in\mathbb{R}^2$ is its position and $\boldsymbol{u}\in\mathbb{R}^2$ is its velocity (\textit{i.e.} the control input).  Assume at a given time, the other robots $\mathcal{O}=\{O_j\}_{j=1}^M$ are located at positions $\{\boldsymbol{x}^o_j\}_{j=1}^{M}$ respectively. 
The ego robot must reach a goal position $\boldsymbol{x}_{d} \in \mathbb{R}^2$ while avoiding collisions with $O_j$ $\forall j\in \{1,2,\cdots,M\}$. The ego robot can use a nominal proportional controller $\hat{\boldsymbol{u}}(\boldsymbol{x})=-k_{p}(\boldsymbol{x}-\boldsymbol{x}_{d})$ that guarantees stabilization towards $\boldsymbol{x}_{d}$. Here $k_{p}>0$ is the controller gain.  For collision avoidance, the ego robot must maintain a distance of at-least $D_s$ with $O_j$ $\forall j\in \{1,2,\cdots,M\}$ \textit{i.e.} their positions $(\boldsymbol{x},\boldsymbol{x}^o_j)$ must satisfy $ \norm{\Delta \boldsymbol{x}_{j}}^2\geq D_s^2$ where $\Delta \boldsymbol{x}_{j} \coloneqq \boldsymbol{x}-\boldsymbol{x}^o_j$ and $D_s$ is a desired safety margin. 

To combine the  collision avoidance requirement with the goal stabilization objective, the ego robot solves a QP that computes a controller closest to the prescribed control $\hat{\boldsymbol{u}}(\boldsymbol{x})=-k_{p}(\boldsymbol{x}-\boldsymbol{x}_{d})$ and satisfies $M$ collision avoidance constraints as follows:
\begin{align}
\label{optimization_formulation_1_static_obstacles}
\begin{aligned}
\boldsymbol{u}^*&= \underset{\boldsymbol{u}}{\arg\min}
& & \norm{\boldsymbol{u} - \hat{\boldsymbol{u}}(\boldsymbol{x})}^2 \\
& \text{subject to}
& & A(\boldsymbol{x})\boldsymbol{u} \leq \boldsymbol{b}(\boldsymbol{x})  
\end{aligned}
\end{align}
Here $A(\boldsymbol{x})\in \mathbb{R}^{M\times 2}$, $\boldsymbol{b}(\boldsymbol{x})\in \mathbb{R}^{M}$ are defined such that the $j^{th}$ row of $A$ is $\boldsymbol{a}^T_j$ and the $j^{th}$ element of $\boldsymbol{b}$ is $b_j$:
\begin{align}
\label{Ab_static}
\boldsymbol{a}^T_j(\boldsymbol{x}) &\coloneqq-\Delta \boldsymbol{x}^T_{j}=-(\boldsymbol{x}-\boldsymbol{x}^o_j)^T \nonumber \\
b_j(\boldsymbol{x}) &\coloneqq \frac{\gamma}{2} (\norm{\Delta \boldsymbol{x}_{j}}^2-D_s^2) \mbox{   }\forall j \in \{1,2,\dots,M\}
\end{align}
The ego robot locally solves this QP at every time step, to determine its optimal control $\boldsymbol{u}^*$, which ensures collision avoidance while encouraging motion towards the goal $\boldsymbol{x}_d$.

Aside from its position $\boldsymbol{x}$, the ego robot's control $\boldsymbol{u}^*$ depends on task parameters defined to be the position of the goal \color{red}$\boldsymbol{x}_d$\color{black}, and the nominal controller's gain \color{red}$k_p$\color{black}.  This is implicitly encoded through the cost function of  \eqref{optimization_formulation_1_static_obstacles} (recall $\hat{\boldsymbol{u}}(\boldsymbol{x})=-\color{red}k_{p}\color{black}(\boldsymbol{x}-\color{red}\boldsymbol{x}_{d}\color{black})$). To highlight this dependence, we denote the control as $\boldsymbol{u}^*_{\boldsymbol{\theta}}(\boldsymbol{x})$ where $\boldsymbol{\theta}$ are the unknown parameters the observer aims to identify.  Next, we formulate the problem that the observer seeks to solve.
\section{Task Identification Problem Formulation}
\label{ProblemFormulation}
The problem for the observer is to monitor the positions of the ego robot $\boldsymbol{x}(t)$ and the ``cooperative obstacles"   $\boldsymbol{x}^o_j(t) \mbox { } \forall j \in \{1,2,\cdots,M\}$ over $t \in [0,T]$ and infer $\boldsymbol{x}_d,k_p$ based on these measurements. Here $T$ is assumed to be large enough.  The observer will run $M+1$ parallel estimators  to identify $\{\boldsymbol{x}^i_d,k^i_p\}_{i=1}^{M+1}$ for each robot in the team. Even though \eqref{original} to \eqref{parameter_update_law} allow for simultaneous  identification of $\boldsymbol{x}_d$ and $k_p$, we focus on decoupled identification  \textit{i.e.} identify (a)  $\boldsymbol{x}_d$ assuming $k_p$ is known ($\boldsymbol{\theta}=\boldsymbol{x}_d$) and (b) $k_p$ assuming $\boldsymbol{x}_d$ is known ($\boldsymbol{\theta}=k_p$) . This is just to keep the analysis simple, joint identification is a straightforward extension.

To estimate $\boldsymbol{\theta}$ using the algorithm in section \ref{AdaptiveObserverAlgorithm} (\eqref{original}- \eqref{parameter_update_law}) and to compute the condition in lemma \ref{lemma1}, the observer requires explicit dynamics of the form $\dot{\boldsymbol{x}}=G(\boldsymbol{x})\boldsymbol{\theta} + f(\boldsymbol{x})$ in \eqref{original}. That is, it must know  $G(\boldsymbol{x})$ and $f(\boldsymbol{x})$. However, owing to the fact that $\dot{\boldsymbol{x}}=\boldsymbol{u}^*_{\boldsymbol{\theta}}(\boldsymbol{x})$ is optimization-based \eqref{optimization_formulation_1_static_obstacles}, such explicit relations are not known. In the next section, we derive the KKT conditions of \eqref{optimization_formulation_1_static_obstacles} which is the first step to derive these expressions. There are a few assumptions first:
\begin{assumption}
	\label{ass1}
	The observer knows the form of safety constraints $A(\boldsymbol{x}),\boldsymbol{b}(\boldsymbol{x})$ in \eqref{optimization_formulation_1_static_obstacles} and that the cost function is of the form $\norm{\boldsymbol{u} - \hat{\boldsymbol{u}}(\boldsymbol{x})}^2$.
\end{assumption}
\begin{assumption}
\label{ass3}
The observer can measure both the position $\boldsymbol{x}(t)$ and velocity (\textit{i.e.} the control $\boldsymbol{u}^*_{\boldsymbol{\theta}}(\boldsymbol{x}(t))$ of the ego robot.
\end{assumption}
Assumption \ref{ass1} is needed since we are interested in deriving the necessary identifiability conditions and they require knowledge of the dynamics. Assumption \ref{ass3}  is not restrictive in practice because positions are easily measurable and velocities can be obtained through numerical differentiation. 
\section{Analysis using KKT conditions}
\label{KKT Conditions}
To analyze the relation between the optimizer of \eqref{optimization_formulation_1_static_obstacles}  \textit{i.e.} $\boldsymbol{u}^*_{\boldsymbol{\theta}}(\boldsymbol{x})$ and parameters $\boldsymbol{\theta}$, we look at the KKT conditions of this QP. These are necessary and sufficient conditions satisfied by $\boldsymbol{u}^*_{\boldsymbol{\theta}}(\boldsymbol{x})$. The Lagrangian for \eqref{optimization_formulation_1_static_obstacles} is
\begin{align}
L(\boldsymbol{u},\boldsymbol{\mu}) =  \norm{\boldsymbol{u} - \hat{\boldsymbol{u}}}^2_2  + \boldsymbol{\mu}^T(A\boldsymbol{u}-\boldsymbol{b}) \nonumber.
\end{align}
Let $(\boldsymbol{u}^*_{\boldsymbol{\theta}},\boldsymbol{\mu}^*_{\boldsymbol{\theta}})$ be the optimal primal-dual solution to  \eqref{optimization_formulation_1_static_obstacles}. The KKT conditions are \cite{boyd2004convex}:
\begin{enumerate}
	\item Stationarity: $\nabla_{\boldsymbol{u}}L(\boldsymbol{u},\boldsymbol{\mu})\vert_{(\boldsymbol{u}^*_{\boldsymbol{\theta}},\boldsymbol{\mu}^*_{\boldsymbol{\theta}})} = 0$,
	\begin{align}
	\label{stationarity1}
		\implies \boldsymbol{u}^*_{\boldsymbol{\theta}} &= \hat{\boldsymbol{u}} - \frac{1}{2}\sum_{j \in \{1,\cdots,M\}}\mu^*_{j{\boldsymbol{\theta}}}\boldsymbol{a}_{j}
	\nonumber \\
    &= \hat{\boldsymbol{u}} - \frac{1}{2} A^T \boldsymbol{\mu}^*_{\boldsymbol{\theta}} .
	\end{align}
\item Primal Feasibility 
\begin{align}
\label{primal_feasibility1}
A\boldsymbol{u}^*_{\boldsymbol{\theta}}\leq \boldsymbol{b} \iff \boldsymbol{a}^T_{j}\boldsymbol{u}^*_{\boldsymbol{\theta}} \leq b_{j} \mbox{ }\forall j \in \{1,\cdots,M\}.
\end{align}
\item Dual Feasibility 
\begin{align}
\label{dual_feasibility1}
{\mu^*_{j{\boldsymbol{\theta}}}} \geq 0 \mbox{  }	 \forall j \in \{1,2,\cdots,M\}.
\end{align}
\item Complementary Slackness 
\begin{align}
	\label{complimentarty slackness1}
	\mu^*_{j{\boldsymbol{\theta}}} \cdot (\boldsymbol{a}^T_{j}\boldsymbol{u}^*_{\boldsymbol{\theta}} -b_{j}) = 0 
	 \mbox{   }\forall j \in \{1,2,\cdots,M\}.
\end{align}
\end{enumerate}
We define the set of active and inactive constraints as
\begin{align}
\label{activeinactive}
	\mathcal{A}(\boldsymbol{u}^*_{\boldsymbol{\theta}}) \coloneqq \{j \in \{1,2,\cdots,M\} \mid \boldsymbol{a}^T_{j}\boldsymbol{u}^*_{\boldsymbol{\theta}} = b_{j} \}, \\
	\mathcal{IA}(\boldsymbol{u}^*_{\boldsymbol{\theta}}) \coloneqq \{j \in \{1,2,\cdots,M\} \mid \boldsymbol{a}^T_{j}\boldsymbol{u}^*_{\boldsymbol{\theta}} < b_{j} \} .
\end{align}
The set of active constraints qualitatively represents those other robots that the ego robot ``worries" about for collisions. From the perspective of the ego robot, we will simply refer to the ``other robots" as \textit{obstacles}. Let there be a total of $K$ active constraints \textit{i.e.}  $\texttt{card}(\mathcal{A}(\boldsymbol{u}^*_{\boldsymbol{\theta}}))=K$ where $K\in \{0,1,\cdots,M\}$. Using \eqref{dual_feasibility1} and \eqref{complimentarty slackness1}, we deduce
\begin{align}
	\mu^*_{j\boldsymbol{\theta}} = 0 \mbox{ $\forall j $} \in \mathcal{IA}(\boldsymbol{u}^*_{\boldsymbol{\theta}}).
\end{align}  
Therefore, we can restrict the summation in \eqref{stationarity1} only to the set of active constraints \textit{i.e.}
\begin{align}
\label{kkt_general}
	\boldsymbol{u}^*_{\boldsymbol{\theta}} &= \hat{\boldsymbol{u}} - \frac{1}{2}\sum_{j \in \mathcal{A}(\boldsymbol{u}^*_{\boldsymbol{\theta}})}\mu^*_{j{\boldsymbol{\theta}}}\boldsymbol{a}_{j}
	\nonumber \\
	&= \hat{\boldsymbol{u}} - \frac{1}{2} A_{ac}^{T} \boldsymbol{\mu}^{ac}_{\boldsymbol{\theta}} .
\end{align}
where $A_{ac}(\boldsymbol{x})\in \mathbb{R}^{K \times 2}$ is the matrix formed using the rows of $A$ that are indexed by the active set $\mathcal{A}(\boldsymbol{u}^*_{\boldsymbol{\theta}})$, and likewise $\boldsymbol{\mu}^{ac}_{\boldsymbol{\theta}}\coloneqq\{\mu^*_{j\boldsymbol{\theta}}\}_{j \in \mathcal{A}(\boldsymbol{u}^*_{\boldsymbol{\theta}})}$. Similarly, let $\boldsymbol{b}_{ac}(\boldsymbol{x})\in \mathbb{R}^{K}$ denote the vector formed from the elements of $\boldsymbol{b}$ indexed by $\mathcal{A}(\boldsymbol{u}^*_{\boldsymbol{\theta}})$. By deleting all inactive constraints and retaining only the active constraints, we can pose another QP that consists only of active constraints, whose solution is the same as that of \eqref{optimization_formulation_1_static_obstacles}. This equality-constrained program (EQP) is given by
\begin{align}
\label{optimization_formulation_2_static_obstacles}
	\begin{aligned}
		\boldsymbol{u}^*&= \underset{\boldsymbol{u}}{\arg\min}
		& & \norm{\boldsymbol{u} - \hat{\boldsymbol{u}}(\boldsymbol{x})}^2 \\
		& \text{subject to}
		& & A_{ac}(\boldsymbol{x})\boldsymbol{u} = \boldsymbol{b}_{ac}(\boldsymbol{x})  
	\end{aligned}
\end{align}
Note that the system $A_{ac}(\boldsymbol{x})\boldsymbol{u} = \boldsymbol{b}_{ac}(\boldsymbol{x})$ is always consistent by construction because of \eqref{activeinactive}, as long as a solution $\boldsymbol{u}^*_{\boldsymbol{\theta}}$ to \eqref{optimization_formulation_1_static_obstacles} exists. 
Now why do we care for this EQP? That is because it is easier to derive an expression  $\boldsymbol{u}^*_{\boldsymbol{\theta}}(\boldsymbol{x})=G(\boldsymbol{x})\boldsymbol{\theta}+f(\boldsymbol{x})$ for  \eqref{optimization_formulation_2_static_obstacles} than the inequality constrained problem \eqref{optimization_formulation_1_static_obstacles}. The only question is how to estimate the active set $\mathcal{A}(\boldsymbol{u}^*_{\boldsymbol{\theta}})$ to determine $A_{ac}(\boldsymbol{x}),b_{ac}(\boldsymbol{x})$ for \eqref{optimization_formulation_2_static_obstacles}. This can be done as follows. From Assumption \ref{ass3}, recall that the observer can measure both the position $\boldsymbol{x}(t)$ and velocity $\boldsymbol{u}^*_{\boldsymbol{\theta}}(\boldsymbol{x}(t))$ of the ego robot. Using these, the observer can determine $\mathcal{A}(\boldsymbol{u}^*_{\boldsymbol{\theta}})$ by comparing the residuals $\vert\boldsymbol{a}^T_{j}(\boldsymbol{x})\boldsymbol{u}^*_{\boldsymbol{\theta}} - b_{j}(\boldsymbol{x})\vert$ against a small threshold $\epsilon>0$ consistent with \eqref{activeinactive}:
\begin{align}
\label{activeobserver}
	\mathcal{A}^{observer} \coloneqq \{j \in \{1,2,\cdots,M\} \mid \vert\boldsymbol{a}^T_{j}(\boldsymbol{x})\boldsymbol{u}^*_{\boldsymbol{\theta}} - b_{j}(\boldsymbol{x})\vert < \epsilon \} \nonumber.
\end{align}
Clearly, computing	$\mathcal{A}^{observer} $ requires the observer to know $A(\boldsymbol{x})$ and $\boldsymbol{b}(\boldsymbol{x})$, hence the need for Assumption \ref{ass1}. At any rate, for a small enough threshold $\epsilon$, it holds true that $\mathcal{A}^{observer}=	\mathcal{A}(\boldsymbol{u}^*_{\boldsymbol{\theta}})$ consistent with \eqref{activeinactive}. This allows the observer to determine the active set.  In the next section, we work with \eqref{optimization_formulation_2_static_obstacles} to derive an explicit expression for control \textit{i.e.} $\boldsymbol{u}^*_{\boldsymbol{\theta}}(\boldsymbol{x})=G(\boldsymbol{x})\boldsymbol{\theta}+f(\boldsymbol{x})$ for various combinations of $\texttt{card}(\mathcal{A}(\boldsymbol{u}^*_{\boldsymbol{\theta}}))=K$ and $\texttt{rank}(A_{ac}(\boldsymbol{x}))$.

\section{SVD based Analysis of $A_{ac}(\boldsymbol{x})\boldsymbol{u} = \boldsymbol{b}_{ac}(\boldsymbol{x})  $}
\label{SVD}
The aim of this section is to derive relations between $\boldsymbol{u}^*_{\boldsymbol{\theta}}$  and $k_p,\boldsymbol{x}_d$ needed for identifying these parameters.  We will show  that the dependence of $\boldsymbol{u}^*_{\boldsymbol{\theta}}$ on these parameters banks on  $\texttt{rank}(A_{ac}(\boldsymbol{x}))$. Theorems \ref{theorem_caseA}, \ref{theorem_caseB} and \ref{theorem_caseC} \textit{roughly} state that whenever there is none or one obstacle for the ego robot to actively avoid, the control $\boldsymbol{u}^*_{\boldsymbol{\theta}}$ exhibits a well-defined dependence on these parameters making their inference using the estimation algorithm in section \ref{AdaptiveObserverAlgorithm} \textit{i.e.} equations \eqref{original}-\eqref{parameter_update_law} possible. On the other hand, theorem \ref{theorem_caseD} states that whenever there are too many obstacles active, the robot is consumed by collision avoidance constraints, so $\boldsymbol{u}^*_{\boldsymbol{\theta}}$ ``gives up" on optimizing the objective. Therefore, $\boldsymbol{u}^*_{\boldsymbol{\theta}}$ does not depend on $k_p,\boldsymbol{x}_d$ making their inference using \eqref{original}-\eqref{parameter_update_law}  impossible.  
\subsection{No active constraints \textit{i.e.} $K=0$}
\label{caseA}
When no constraint is active, we have $\mu_{j\boldsymbol{\theta}}=0 \mbox{ } \forall j \in \{1,2,\cdots,M\}$, so from \eqref{stationarity1} we get $\boldsymbol{u}^*_{\boldsymbol{\theta}}=\hat{\boldsymbol{u}}(\boldsymbol{x})=-k_p(\boldsymbol{x}-\boldsymbol{x}_d)$. Intuitively this means that the robot does not worry about collisions with any obstacle, so it is free to use $\hat{\boldsymbol{u}}$ itself. We write this in the parameter affine form \textit{i.e.} $\dot{\boldsymbol{x}}=\boldsymbol{u}^*_{\boldsymbol{\theta}}=-k_p(\boldsymbol{x}-\boldsymbol{x}_d) = G(\boldsymbol{x})\boldsymbol{\theta}+f(\boldsymbol{x})$ and derive conditions under which estimation is possible.
\begin{theorem}
\label{theorem_caseA}
If $\forall t \in [0,T]$, no constraint is active, then the observer can always estimate the goal using $\boldsymbol{x}(t),\boldsymbol{u}^*(\boldsymbol{x}(t)) \mbox{ } \forall t \in [0,T]$, assuming gain is known. Likewise, the observer can always estimate gain assuming goal is known, as long as the robot is not already at its goal.
\end{theorem}
\begin{proof}
\begin{enumerate}[label=(\alph*)]
    \item If the observer wants to estimate the goal \textit{i.e.} $\boldsymbol{\theta}=\boldsymbol{x}_d$, then defining $G(\boldsymbol{x})\coloneqq k_pI$ and $f(\boldsymbol{x})\coloneqq -k_p\boldsymbol{x}$ gives $\boldsymbol{u}^*_{\boldsymbol{x}_d}=\hat{\boldsymbol{u}}(\boldsymbol{x})=G(\boldsymbol{x})\boldsymbol{x}_d+f(\boldsymbol{x})$. For this case, goal estimation is always possible because  $G(\boldsymbol{x})^TG(\boldsymbol{x})=k_p^2I\succ 0$ \textit{i.e.} $G(\boldsymbol{x}(t))$ is \textbf{PE} (and \textbf{IE} \eqref{GisPE}) and $G(\boldsymbol{x})$ has no null space (lemma \ref{lemma1}).
    \item If the observer wants to estimate thr gain \textit{i.e.} $\boldsymbol{\theta}=k_p$, defining $G(\boldsymbol{x})\coloneqq -(\boldsymbol{x}-\boldsymbol{x}_d)$ and $f(\boldsymbol{x})\coloneqq \boldsymbol{0}$ gives $\boldsymbol{u}^*_{k_p}=G(\boldsymbol{x})k_p+f(\boldsymbol{x})$. Gain estimation is only possible when $G^T(\boldsymbol{x}(t))G(\boldsymbol{x}(t))=\norm{\boldsymbol{x}(t)-\boldsymbol{x}_d}^2 \neq 0 \forall t\in[0,T]$ \textit{i.e.} when the robot is not at its goal. This is expected because if the robot is already at its goal, then it will stay there forever, so there is no information in its positions about $k_p$, hence the result.  \vspace{-0.45cm}
\end{enumerate}
\end{proof}
\subsection{Exactly one active constraint \textit{i.e.} $K=1$}
\label{caseB}
When one constraint is active, there is one obstacle that the ego robot ``worries" about for collision. Since there are two degrees of freedom in the control, and one obstacle to avoid, the ego robot can avoid this obstacle and additionally minimize $\norm{\boldsymbol{u}-\hat{\boldsymbol{u}}(\boldsymbol{x})}^2$ with the remaining degree of freedom.  This causes $\boldsymbol{u}^*_{\boldsymbol{\theta}}$ to exhibit a well-defined dependence on $\hat{\boldsymbol{u}}(\boldsymbol{x})$  and by extension, on parameters $k_p,\boldsymbol{x}_d$. This  makes their inference using the estimation algorithm in section \ref{AdaptiveObserverAlgorithm} \textit{i.e.} equations \eqref{original}-\eqref{parameter_update_law} feasible.

Let $i \in \{1,2,\cdots,M\}$ denote the index of the active constraint, meaning that it is the obstacle located at $\boldsymbol{x}_i^o$ that should be ``actively" avoided. Thus, from \eqref{activeinactive}, we have $A_{ac}(\boldsymbol{x})\boldsymbol{u}^*_{\boldsymbol{\theta}}=\boldsymbol{a}^T_i(\boldsymbol{x})\boldsymbol{u}^*_{\boldsymbol{\theta}}=b_i(\boldsymbol{x})$  where $A_{ac}(\boldsymbol{x})\coloneqq \boldsymbol{a}^T_i(\boldsymbol{x}) $  and $\boldsymbol{a}^T_i(\boldsymbol{x}),
\boldsymbol{b}_i(\boldsymbol{x})$ are defined in  \eqref{Ab_static}. Since $\boldsymbol{a}^T_i(\boldsymbol{x}) \in \mathbb{R}^{1\times 2}$, from rank-nullity theorem it follows that $\boldsymbol{a}^T_i(\boldsymbol{x})$ has a non-trivial null space of dimension one. The null space gives a degree of freedom to the control to minimize $\norm{\boldsymbol{u} - \hat{\boldsymbol{u}}(\boldsymbol{x})}^2$ while satisfying the constraint. We illustrate this by computing the SVD $\boldsymbol{a}^T_i(\boldsymbol{x}) =U(\boldsymbol{x})\Sigma(\boldsymbol{x}) V^T(\boldsymbol{x})$. Defining
\begin{align}
U(\boldsymbol{x}) &\coloneqq 1 \nonumber \\
\Sigma(\boldsymbol{x}) &\coloneqq \left[\begin{matrix} \Sigma_m(\boldsymbol{x}),0 \end{matrix}\right] \mbox{ where }\Sigma_m(\boldsymbol{x})=\norm{\boldsymbol{a}_i(\boldsymbol{x})} \nonumber \\
V(\boldsymbol{x}) &\coloneqq  \left[\begin{matrix} V_1,V_2 \end{matrix}\right] \mbox{ } V_1=\frac{\boldsymbol{a}_i(\boldsymbol{x})}{\norm{\boldsymbol{a}_i(\boldsymbol{x})}},V_2=R_{\frac{\pi}{2}}\frac{\boldsymbol{a}_i(\boldsymbol{x})}{\norm{\boldsymbol{a}_i(\boldsymbol{x})}}.   
\end{align} 
Since $V$ forms a basis for $\mathbb{R}^2$, any $\boldsymbol{u}$ can be expressed as
\begin{align}
\boldsymbol{u}&=\left[\begin{matrix} V_1,V_2 \end{matrix}\right] \left[\begin{matrix} \tilde{u}_1 \\ \tilde{u}_2\end{matrix}\right]  \nonumber.  \\
\implies \boldsymbol{a}^T_i\boldsymbol{u}-b_i &=U\left[\begin{matrix} \Sigma_m,0 \end{matrix}\right] \left[\begin{matrix} V_1^T \\ V_2^T \end{matrix}\right] \left[\begin{matrix} V_1,V_2 \end{matrix}\right] \left[\begin{matrix} \tilde{u}_1 \\ \tilde{u}_2\end{matrix}\right]-b_i  \nonumber \\
&=U\Sigma_m\tilde{u}_1+ 0\cdot\tilde{u}_2-b_i=0
\end{align}
Choosing $\tilde{u}_1 = \Sigma_m^{-1}U^Tb_i$ and $\tilde{u}_2 = \psi \in \mathbb{R}$, we find that  
\begin{align}
\label{control_case2_general}
\boldsymbol{u}= V_1\Sigma_m^{-1}U^Tb_i + V_2\psi
\end{align}
satisfies $\boldsymbol{a}^T_i(\boldsymbol{x})\boldsymbol{u}=b_i(\boldsymbol{x})$ $ \forall \psi \in \mathbb{R}$. Recall from the properties of SVD that $V_2$ forms a basis for $\mathcal{N}(\boldsymbol{a}^T_i(\boldsymbol{x}))$. We tune $\psi$ to minimize $\norm{\boldsymbol{u} - \hat{\boldsymbol{u}}}^2$ by solving the following unconstrained minimization problem 
\begin{align}
\label{psi_determine}
\begin{aligned}
\psi^*&= \underset{\psi}{\arg\min}
& &\norm{\boldsymbol{u} - \hat{\boldsymbol{u}}}^2 \\
&= \underset{\psi}{\arg\min}  
& &\norm{V_1\Sigma_m^{-1}U^Tb_i + V_2\psi -\hat{\boldsymbol{u}} }^2,
\end{aligned}
\end{align}
which gives $\psi^*=V_2^T \hat{\boldsymbol{u}}$. Substituting this in \eqref{control_case2_general}, gives
\begin{align}
\label{control_case2_optimal}
\boldsymbol{u}^*_{\boldsymbol{\theta}} = V_1\Sigma_m^{-1}U^Tb_i + V_2V_2^T\hat{\boldsymbol{u}}.
\end{align} 
This equation is the solution to \eqref{optimization_formulation_2_static_obstacles} and by extension, to \eqref{optimization_formulation_1_static_obstacles}. Note that it is the second term $V_2V_2^T\hat{\boldsymbol{u}}$ that depends on  parameters  $k_p,\boldsymbol{x}_d$ because $\hat{\boldsymbol{u}} = -k_p(\boldsymbol{x}-\boldsymbol{x}_d)$. Using this relation \eqref{control_case2_optimal}, we are ready to state the conditions under which inference of parameters using the algorithm ( \eqref{original}-\eqref{parameter_update_law}) and lemma \ref{lemma1}  in section \ref{AdaptiveObserverAlgorithm} is possible.
\begin{theorem}
	\label{theorem_caseB}
	If $\forall t \in [0,T]$, exactly one constraint is active, then the observer can estimate the goal (gain) using $\boldsymbol{x}(t),\boldsymbol{u}^*(\boldsymbol{x}(t)) \mbox{ } \forall t \in [0,T]$, assuming the gain (goal) is known, as long as the orientation of $\frac{\boldsymbol{a}_i(\boldsymbol{x})}{\norm{\boldsymbol{a}_i(\boldsymbol{x})}}$ is not time-invariant and $\boldsymbol{x}(t)\neq \boldsymbol{x}_d\mbox{ } \forall t \in [0,T]$ 
\end{theorem}
\begin{proof}
	\begin{enumerate}[label=(\alph*)]
		\item If the observer wants to estimate the goal \textit{i.e.} $\boldsymbol{\theta}=\boldsymbol{x}_d$, then define $G(\boldsymbol{x})\coloneqq k_pV_2V_2^T$ and $f(\boldsymbol{x})\coloneqq V_1\Sigma_m^{-1}U^Tb_i - k_pV_2V_2^T\boldsymbol{x}$ using \eqref{control_case2_optimal} so that $\boldsymbol{u}^*_{\boldsymbol{x}_d}=G(\boldsymbol{x})\boldsymbol{x}_d+f(\boldsymbol{x})$. The \textbf{IE} condition (Def. \ref{IE}, \eqref{GisPE})  requires that $\int_{0}^T G^T(\boldsymbol{x}(t))G(\boldsymbol{x}(t))dt \succ 0$. The situation when positive-definiteness is not attained is when $\mathcal{N}(G(\boldsymbol{x}(t))$ is time-invariant which follows from Lemma \ref{lemma1}. Note that $\mathcal{N}(G(\boldsymbol{x}))=\mathcal{N}(k_pV_2V_2^T)=V_1=\frac{\boldsymbol{a}_i(\boldsymbol{x})}{\norm{\boldsymbol{a}_i(\boldsymbol{x})}}$ which follows from the properties of SVD. Since $V_1$ is always a unit vector, it can only change through its orientation. If its orientation does not change over $[0,T]$, then $V_1$ is a time-invariant vector in $\mathcal{N}(G(\boldsymbol{x}))$ and hence goal estimation will not be possible, using Lemma \ref{lemma1}. 	

		\item If $\boldsymbol{\theta}=k_p$, then $G(\boldsymbol{x})\coloneqq -V_2V_2^T(\boldsymbol{x}-\boldsymbol{x}_d)$ and $f(\boldsymbol{x})\coloneqq V_1\Sigma_m^{-1}U^Tb_i $ so that $\boldsymbol{u}^*_{k_p}=G(\boldsymbol{x})k_p+f(\boldsymbol{x})$. If $(\boldsymbol{x}-\boldsymbol{x}_d) \parallel V_1$ $\implies$ $(\boldsymbol{x}-\boldsymbol{x}_d) \perp V_2$ then $G(\boldsymbol{x}) \equiv \boldsymbol{0}$. So if $(\boldsymbol{x}(t)-\boldsymbol{x}_d) \parallel V_1(\boldsymbol{x}(t)) \mbox{  } \forall t \in [0,T] $ then the \textbf{IE} condition (Def. \ref{IE})  for gain identification will not be satisfied.  \vspace{-0.45cm}
	\end{enumerate} 
\end{proof}
The video at  \url{https://youtu.be/WoUSej79ZGM} shows an example where invariance of the orientation of null-space results in failure to identify goal using the estimation algorithm in \ref{AdaptiveObserverAlgorithm} \eqref{original}-\eqref{parameter_update_law} and also a UKF.
\subsection{$2\leq K\leq M$ and $\texttt{rank}(A_{ac}(\boldsymbol{x}))=1$}
\label{caseE}
Now we consider the more general case in which there is more than one constraint active, but all of these are linearly dependent on one constraint among them. This means that effectively there is only one ``representative constraint" or obstacle for the ego robot to worry about. Consequently, this case is similar to the case with just one active obtacle. We formally demontrate this now. Let $i_1,i_2,\cdots,i_K\in\{1,\cdots,M\}$ be the indices of active constraints which satisfy
\begin{align}
\label{abK}
\left[\begin{matrix}
\boldsymbol{a}^T_{i_1}(\boldsymbol{x}) 
\\ \boldsymbol{a}^T_{i_2}(\boldsymbol{x}) \\
\vdots \\
\boldsymbol{a}^T_{i_K}(\boldsymbol{x}) 
\end{matrix}\right]\boldsymbol{u}^*_{\boldsymbol{\theta}}&= 
\left[\begin{matrix}b_{i_1}(\boldsymbol{x}) \\b_{i_2}(\boldsymbol{x}) \\ \vdots \\ b_{i_K}(\boldsymbol{x}) \end{matrix}\right]  \mbox{ or } \nonumber \\
A_{ac}(\boldsymbol{x})\boldsymbol{u}^*_{\boldsymbol{\theta}}&=\boldsymbol{b}_{ac}(\boldsymbol{x}),
\end{align}
where $\boldsymbol{a}^T_{i_j}(\boldsymbol{x})  \mbox{ and } b_{i_j}(\boldsymbol{x})$ are defined using \eqref{Ab_static}.  Since $\texttt{rank}(A_{ac}(\boldsymbol{x}))=1$, WLOG we have  $\boldsymbol{a}^T_{i_j}(\boldsymbol{x})=\lambda_j\boldsymbol{a}^T_{i_1}(\boldsymbol{x})$ where $ \lambda_j \in \mathbb{R}\mbox{ }\forall j \in \{2,3,\cdots,K\}$.  Let's first see the geometric arrangements of the robot and the obstacles $i_1,i_2\cdots,i_K$ when this case arises in practice.
\begin{lemma}
	\label{lemma4}
	The case with more than one constraint active and all linearly dependent can only arise in practice for $\lambda_j\in \{+1,-1\}$. $\lambda_j=+1$ means that obstacles indexed $i_1$ and $i_j$ are coinciding and $b_{i_1}(\boldsymbol{x})=b_{i_j}(\boldsymbol{x})$.  $\lambda_j=-1$ means that the robot is located exactly in the middle of obstacles $i_1$ and $i_j$. Furthermore, even if there is just one $j$ for which $\lambda_j=-1$, then $b_{i_j}(\boldsymbol{x})=0 \mbox{ } \forall j \in \{1,2,\cdots,K\} \iff \boldsymbol{b}_{ac}(\boldsymbol{x})=\boldsymbol{0}$. 
\end{lemma}
\begin{proof}
	Since $i_1,i_j$ are active constraints $\forall j \in \{2,3,\cdots,K\}$
	\begin{align}
	\label{ijactive_K}
	\boldsymbol{a}^T_{i_1}(\boldsymbol{x})\boldsymbol{u}^*_{\boldsymbol{\theta}}&=b_{i_1}(\boldsymbol{x}) \\
	\boldsymbol{a}^T_{i_j}(\boldsymbol{x})\boldsymbol{u}^*_{\boldsymbol{\theta}}&=b_{i_j}(\boldsymbol{x}) \mbox{  } \forall j \in \{2,3,\cdots,K\}
	\end{align}
	Substituting  $\boldsymbol{a}^T_{i_1}=\lambda \boldsymbol{a}^T_{i_j}$ in \eqref{ijactive_K}, we get
	\begin{align}
	\label{beqn_K}
	\lambda \boldsymbol{a}^T_{i_j}(\boldsymbol{x})\boldsymbol{u}^*_{\boldsymbol{\theta}}&=b_{i_1}(\boldsymbol{x}) \nonumber \\
	\implies \lambda b_{i_j}(\boldsymbol{x}) &= b_{i_1}(\boldsymbol{x}) \mbox{ }\forall j \in \{2,3,\cdots,K\}
	\end{align}
	Recalling that $\boldsymbol{b}_{r}(\boldsymbol{x}) \coloneqq \frac{\gamma}{2} (\norm{\boldsymbol{a}_{r}}^2-D_s^2) $ from \eqref{Ab_static}, we get 
	\begin{align}
	&\lambda \frac{\gamma}{2} (\norm{\boldsymbol{a}_{i_j}}^2-D_s^2)= \frac{\gamma}{2} (\norm{\boldsymbol{a}_{i_1}}^2-D_s^2)  \nonumber \\
	\implies &\lambda (\norm{\boldsymbol{a}_{i_j}}^2-D_s^2) = (\norm{\lambda \boldsymbol{a}_{i_j}}^2-D_s^2)  \nonumber \\
	\implies &\lambda^2\norm{\boldsymbol{a}_{i_j}}^2 - \lambda (\norm{\boldsymbol{a}_{i_j}}^2-D_s^2) -D_s^2 =0
	\end{align}
	This equation has two roots $\lambda=1,-\frac{D_s^2}{\norm{\boldsymbol{a}_{i_j}}^2}$. 
	\begin{enumerate}
		\item $\lambda=1 \implies b_{i_1}(\boldsymbol{x})=b_{i_j}(\boldsymbol{x})$ and $\boldsymbol{a}^T_{i_1}=\boldsymbol{a}^T_{i_j}$ \textit{i.e.} $\boldsymbol{x}-\boldsymbol{x}^o_{i_1}=\boldsymbol{x}-\boldsymbol{x}^o_{i_j}$ or $\boldsymbol{x}^o_{i_1}=\boldsymbol{x}^o_{i_j}$. This means obstacle $i_j $ is coinciding with obstacle $i_1$.  This is a trivial yet an expected result.
		\item $\lambda=-\frac{D_s^2}{\norm{\boldsymbol{a}_{i_j}}^2}<0$ implies that $\boldsymbol{a}^T_{i_1},\boldsymbol{a}^T_{i_j}$ are anti-parallel. However, when $\lambda<0$, $b_{i_j}(\boldsymbol{x})>0 \implies b_{i_1}(\boldsymbol{x})<0$. Recalling the definition of $b_{i_j}(\boldsymbol{x})$, we know that $b_{i_j}(\boldsymbol{x})>0 \iff \norm{\boldsymbol{a}_{i_j}(\boldsymbol{x})}^2>D_s^2  $ and therefore $b_{i_1}(\boldsymbol{x})<0 \iff \norm{\boldsymbol{a}_{i_1}(\boldsymbol{x})}^2<D_s^2$.  This means that if the robot is \textit{strictly safe} with respect to obstacle $i_j$, then it is colliding with obstacle $i_1$. This means that the control at the previous time step $\boldsymbol{u}^*_{\boldsymbol{\theta}}(\boldsymbol{x}(t^-))$ caused this collision which is not possible. This conflict can only be resolved when we relax \textit{strict safety} to $b_{i_j}(\boldsymbol{x})=0 \implies \norm{\boldsymbol{a}_{i_j}(\boldsymbol{x})}^2=D_s^2 $ meaning that the robot and obstacle $j$ are touching each other. This gives $b_{i_1}(\boldsymbol{x})=0$ implying that the robot and obstacle $i_1$ are also touching each other. In this case $\lambda=-\frac{D_s^2}{\norm{\boldsymbol{a}_{i_j}}^2}=-1$ which implies  $\boldsymbol{x}-\boldsymbol{x}^o_{i_1}=-(\boldsymbol{x}-\boldsymbol{x}^o_{i_j})$ or $\boldsymbol{x}=\frac{1}{2}(\boldsymbol{x}^o_{i_1}+\boldsymbol{x}^o_{i_j})$. Furthermore, even if there is one $j$ for which $\lambda_j=0$, then from \eqref{beqn_K}, $b_{i_j}=0 \forall j \{1,2,\cdots,,K\} \implies \boldsymbol{b}_{ac}=\boldsymbol{0}$  \vspace{-0.45cm}
	\end{enumerate} 
\end{proof}
Next, we derive an analytical expression for $\boldsymbol{u}^*_{\boldsymbol{\theta}}(\boldsymbol{x})$ using \eqref{abK}. Note $A_{ac}(\boldsymbol{x})=U(\boldsymbol{x})\Sigma(\boldsymbol{x})V^T(\boldsymbol{x})$ where
\begin{align}
\label{svdCaseB}
U&\coloneqq  \left[\begin{matrix} U_1,U_2 \end{matrix}\right] \mbox{ , }  U_1=\frac{1}{\sqrt{1+\Sigma_{j=2}^K\lambda_j^2}} \left[\begin{matrix} 1,  \lambda_2, \cdots  \lambda_K \end{matrix}\right]^T  \nonumber \\
\Sigma &\coloneqq 
\left[\begin{array}{@{}c|c@{}}
\Sigma_r& 0^{1 \times 1} \\
\hline
0^{K-1 \times 1}  &0^{K-1 \times 1} 
\end{array}\right] 
\mbox{ , }\Sigma_r=\sqrt{1+ \Sigma_{j=2}^K\lambda_j^2 }\norm{\boldsymbol{a}_i(\boldsymbol{x})} \nonumber \\
V&\coloneqq  \left[\begin{matrix} V_1,V_2 \end{matrix}\right] \mbox{,} V_1=\frac{\boldsymbol{a}_{i_1}(\boldsymbol{x})}{\norm{\boldsymbol{a}_{i_1}(\boldsymbol{x})}},V_2=R_{\frac{\pi}{2}}\frac{\boldsymbol{a}_{i_1}(\boldsymbol{x})}{\norm{\boldsymbol{a}_{i_1}(\boldsymbol{x})}}. 
\end{align}
Choosing $\boldsymbol{u}=V_1\tilde{u}_1+V_2\tilde{u}_2$, from \eqref{abK} we get
\begin{align}
A_{ac}\boldsymbol{u}-\boldsymbol{b}_{ac} &= \left[\begin{matrix} U_1,U_2 \end{matrix}\right]   \left[\begin{array}{@{}c|c@{}}
\Sigma_r & 0 \\
\hline
0  &0
\end{array}\right]  \left[\begin{matrix} V^T_1 \\ V^T_2 \end{matrix}\right] \left[\begin{matrix} V_1,V_2 \end{matrix}\right]  \left[\begin{matrix} \tilde{u}_1 \\ \tilde{u}_2 \end{matrix}\right] -\boldsymbol{b}_{ac} \nonumber \\
&= \left[\begin{matrix} U_1,U_2 \end{matrix}\right] \bigg(\left[\begin{matrix} \Sigma_r\tilde{u}_1 \\0 \end{matrix}\right]  -\left[\begin{matrix} U^T_1\boldsymbol{b}_{ac}\\U^T_2\boldsymbol{b}_{ac} \end{matrix}\right]\bigg).
\end{align}
Since $\norm{}^2$ is unitary invariant, from \eqref{abK}
\begin{align}
\norm{A_{ac}\boldsymbol{u}-\boldsymbol{b}_{ac} }^2 & = \norm{\left[\begin{matrix} \Sigma_r\tilde{u}_1 \\0 \end{matrix}\right]  -\left[\begin{matrix} U^T_1\boldsymbol{b}_{ac}\\U^T_2\boldsymbol{b}_{ac} \end{matrix}\right]}^2 \nonumber \\
&= \norm{\Sigma_r\tilde{u}_1 -U^T_1\boldsymbol{b}_{ac} }^2 + \norm{U^T_2\boldsymbol{b}_{ac}}^2 .
\end{align}
The minimum norm is achieved for $\tilde{u}_1=\Sigma_r^{-1}U_1^T\boldsymbol{b}_{ac}$. Choosing $\tilde{u}_2 =\psi \in \mathbb{R}$, the ``least-squares" solutions are
\begin{align}
\boldsymbol{u}&=V_1\Sigma_r^{-1}U_1^T\boldsymbol{b}_{ac} + V_2\psi.
\end{align}
Computing $\psi$ by minimizing $\norm{\boldsymbol{u} - \hat{\boldsymbol{u}}}^2$, we get $\psi^*=V_2^T\hat{\boldsymbol{u}}$ which gives 
\begin{align}
\label{control_case_3_optimal}
\boldsymbol{u}^*_{\boldsymbol{\theta}}&=V_1\Sigma_r^{-1}U_1^T\boldsymbol{b}_{ac} + V_2V_2^T\hat{\boldsymbol{u}}.
\end{align}
Suppose  $\lambda_j=+1 \mbox{ } \forall j \in \{2,3,\cdots,K\}$, then from \eqref{svdCaseB}, we have $U_1= \frac{1}{\sqrt{K}}\mathbf{1}$, $\Sigma_r= \sqrt{K} \norm{\boldsymbol{a}_{i_1}(\boldsymbol{x})}$. Moreover, from lemma \ref{lemma4}, we have $b_{i_j}(\boldsymbol{x})=b_{i_1}(\boldsymbol{x}) \mbox{ } \forall j  \in \{2,3,\cdots,K\}$, this means that, and $ \boldsymbol{b}_{ac}(\boldsymbol{x})=\mathbf{1}b_{i_1}(\boldsymbol{x})$. Substituting this in \eqref{control_case_3_optimal}, one can verify that we get the same expression for control as in \eqref{control_case2_optimal}. This is  expected because $\lambda_j=1 \mbox{ } \forall j \in \{2,3,\cdots,K\}$ means that all obstacles are coinciding so the ego robot treats them all as just one obstacle, hence the control is identical to one when there was just one active obstacle in \ref{caseB}.  The slight difference between this case and \ref{caseB} comes when there is a $j$ for which  $\lambda_j=-1$. From lemma \ref{lemma4}, this happens when the robot is in the middle of obstacles $i_1$ and $i_j$ and $\boldsymbol{b}_{ac}=\boldsymbol{0}$. Then it follows from \eqref{control_case_3_optimal} that
\begin{align}
\boldsymbol{u}^*_{\boldsymbol{\theta}}&= V_2V_2^T\hat{\boldsymbol{u}}
\end{align}
$V_2V_2^T\hat{\boldsymbol{u}}$ is the projection of  $\hat{\boldsymbol{u}}$ along $V_2=R_{\frac{\pi}{2}}\frac{\boldsymbol{a}_{i_1}(\boldsymbol{x})}{\norm{\boldsymbol{a}_{i_1}(\boldsymbol{x})}}$. This is expected because when the robot is in the middle of the obstacles (lemma \ref{lemma4}), the only feasible direction of motion is along the line that is perpendicular to the line segment connecting the obstacles \textit{i.e.} along $R_{\frac{\pi}{2}}\frac{\boldsymbol{a}_{i_1}(\boldsymbol{x})}{\norm{\boldsymbol{a}_{i_1}(\boldsymbol{x})}}$ because motion along any other direction will cause collisions. For any $\lambda$, $\boldsymbol{u}^*_{\boldsymbol{\theta}}$ in \eqref{control_case_3_optimal} depends on $k_p,\boldsymbol{x}_d$ because of $\hat{\boldsymbol{u}}$.  We next state the conditions under which their inference  s possible.
\begin{theorem}
	\label{theorem_caseC}
	If $\forall t \in [0,T]$, two or more than two constraints are active, all of which are linearly dependent on one among them, then the observer can estimate the goal (gain) using $\boldsymbol{x}(t),\boldsymbol{u}^*(\boldsymbol{x}(t)) \mbox{ } \forall t \in [0,T]$, assuming the gain (goal) is known, as long as the orientation of $\frac{\boldsymbol{a}_{i_1}(\boldsymbol{x})}{\norm{\boldsymbol{a}_{i_1}(\boldsymbol{x})}}$ is not time-invariant and $\boldsymbol{x}(t)\neq \boldsymbol{x}_d\mbox{ } \forall t \in [0,T]$ 
\end{theorem}
\begin{proof}
	\begin{enumerate}[label=(\alph*)]
		\item If $\boldsymbol{\theta}=\boldsymbol{x}_d$, then define $G(\boldsymbol{x})\coloneqq k_pV_2V_2^T$ and $f(\boldsymbol{x})\coloneqq V_1\Sigma_r^{-1}U_1^T\boldsymbol{b}_{ac} - k_pV_2V_2^T\boldsymbol{x}$ using \eqref{control_case_3_optimal} so that $\boldsymbol{u}^*_{\boldsymbol{x}_d}=G(\boldsymbol{x})\boldsymbol{x}_d+f(\boldsymbol{x})$. The \textbf{IE} condition (Def. \ref{IE} and \eqref{GisPE}) requires that the orientation of $\mathcal{N}(G(\boldsymbol{x}))=\frac{\boldsymbol{a}_{i_1}(\boldsymbol{x})}{\norm{\boldsymbol{a}_{i_1}(\boldsymbol{x})}}$ must not stay invariant over $[0,T]$ for $\boldsymbol{x}_d$ estimation to be possible.
		\item If $\boldsymbol{\theta}=k_p$, then $G(\boldsymbol{x})\coloneqq -V_2V_2^T(\boldsymbol{x}-\boldsymbol{x}_d)$ and $f(\boldsymbol{x})\coloneqq V_1\Sigma_r^{-1}U_1^T\boldsymbol{b}_{ac} $ so that $\boldsymbol{u}^*_{k_p}=G(\boldsymbol{x})k_p+f(\boldsymbol{x})$.  If $(\boldsymbol{x}-\boldsymbol{x}_d) \parallel V_1$ $\implies$ $(\boldsymbol{x}-\boldsymbol{x}_d) \perp V_2$ then $G(\boldsymbol{x}) \equiv \boldsymbol{0}$. So if $(\boldsymbol{x}(t)-\boldsymbol{x}_d) \parallel V_1(\boldsymbol{x}(t)) \mbox{  } \forall t \in [0,T] $ then the \textbf{IE} condition (Def. \ref{IE})  for gain identification will not be satisfied. \vspace{-0.42cm}
	\end{enumerate} 
\end{proof}

\begin{figure*}
	\centering     
	\subfigure[$t=0.02s$]{\label{fig:a1}\includegraphics[width=0.24\textwidth]{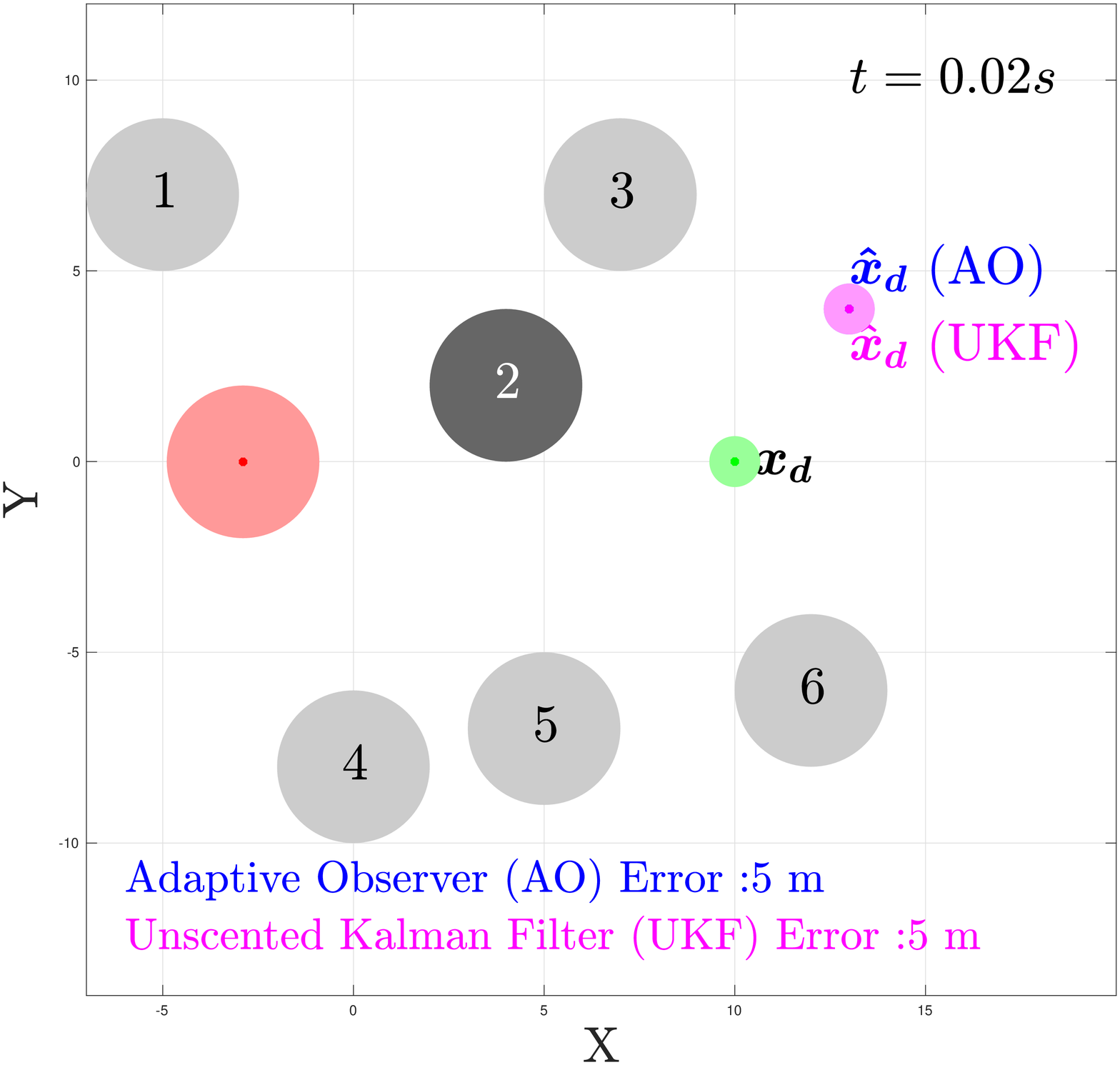}}
	\subfigure[$t=1s$]{\label{fig:b1}\includegraphics[width=0.24\textwidth]{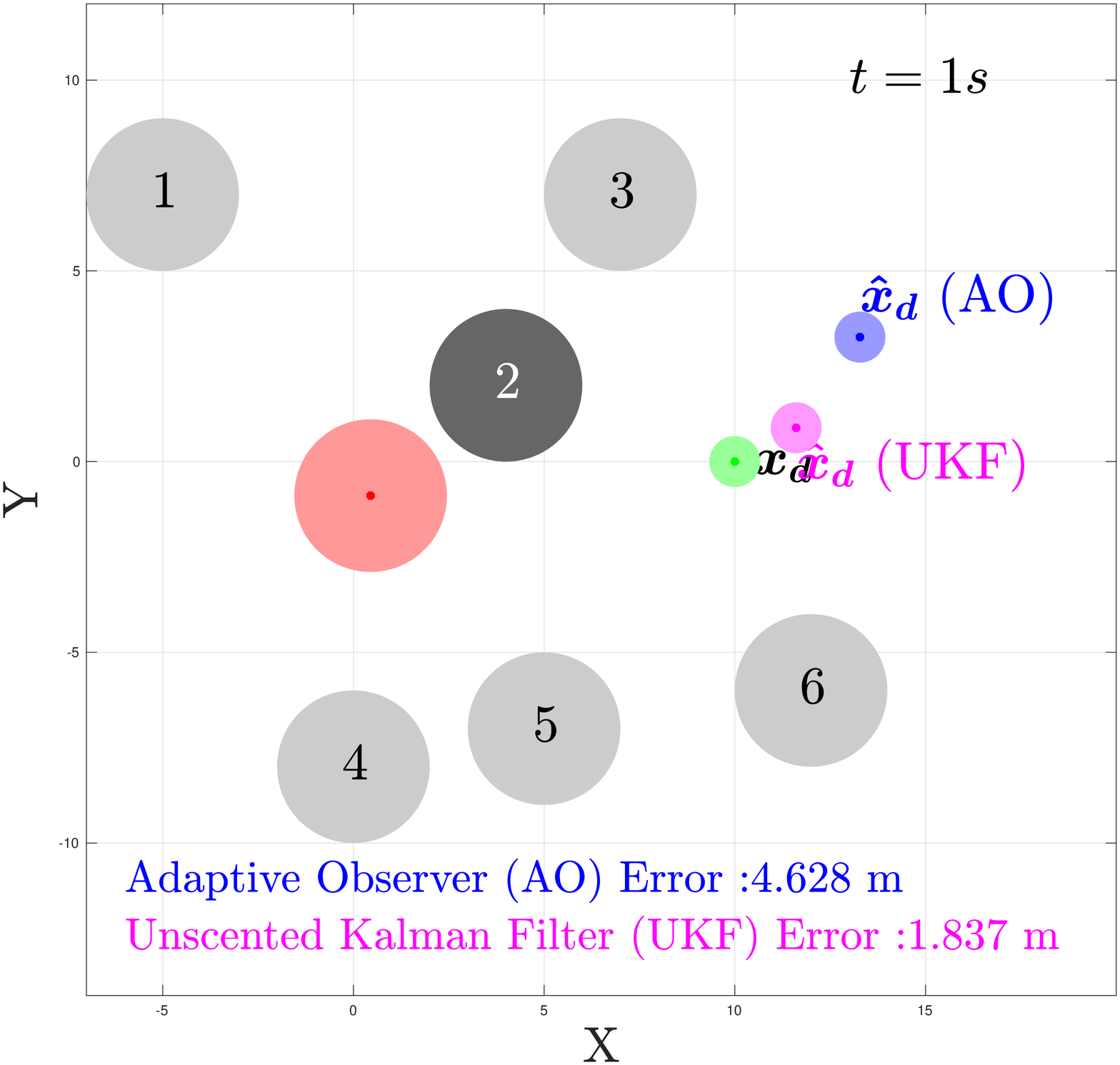}}
	\subfigure[$t=3s$]{\label{fig:c1}\includegraphics[width=0.24\textwidth]{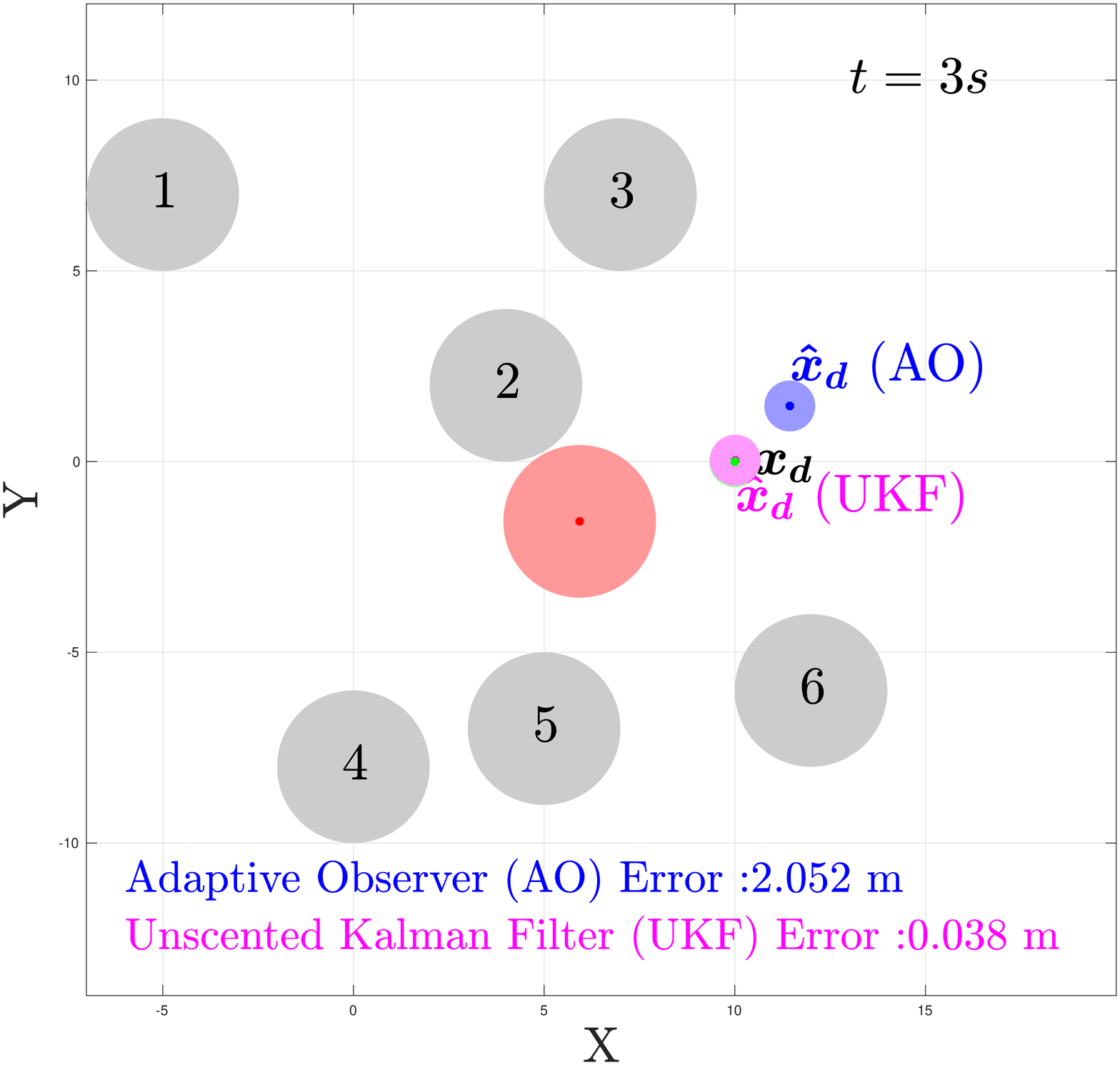}}
	\subfigure[Adaptive Observer and UKF]{\label{fig:d1}\includegraphics[width=0.24\textwidth]{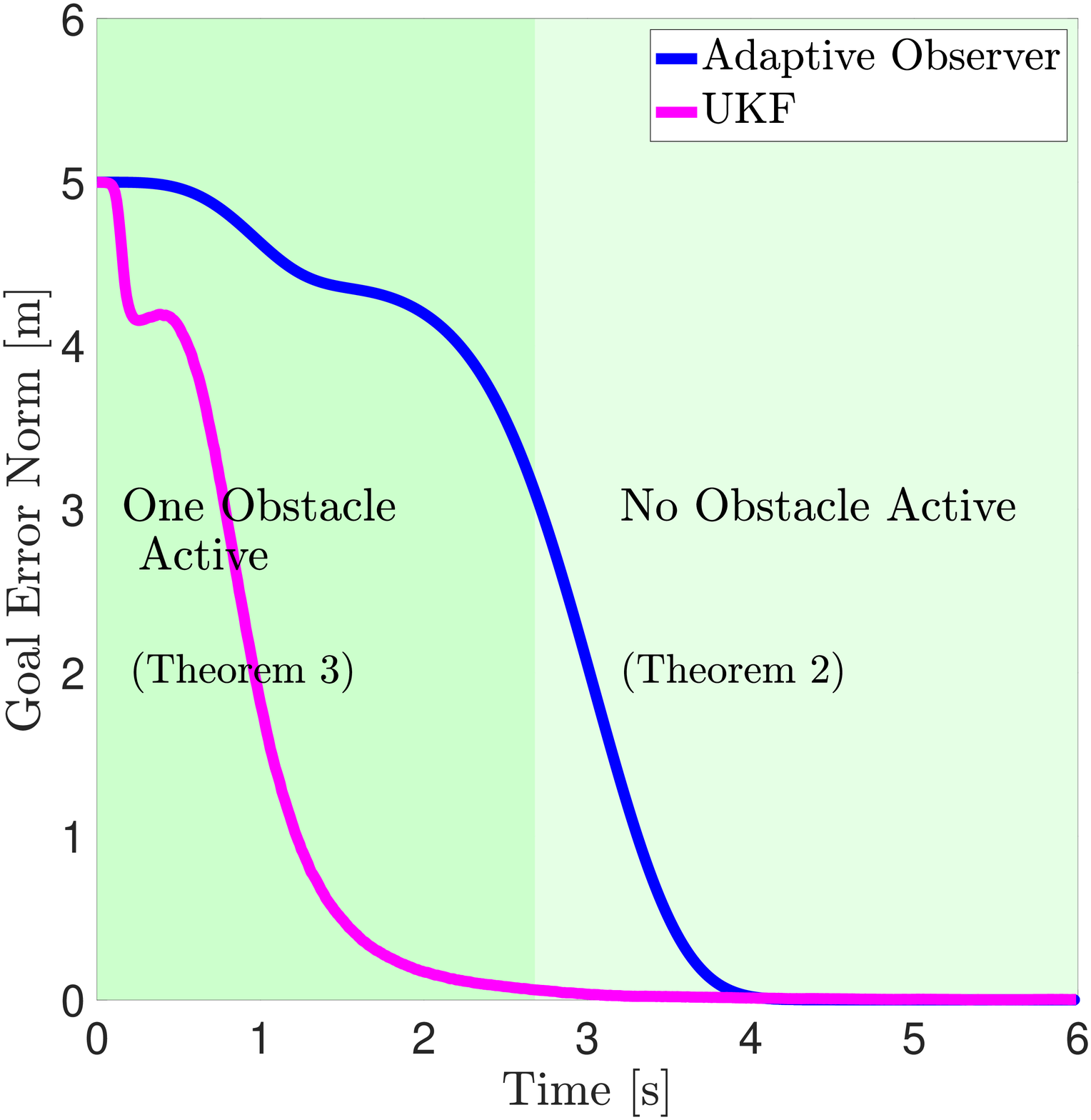}}
	\caption{(a)-(c) A robot navigating to its goal (green) amongst static obstacles. Dark disc represents an active obstacle. Estimates of goal using AO and UKF are shown in blue and pink discs. (d) The norms of goal position errors for AO and UKF. Since atmost one obstacle is active, estimation errors converge to zero. Video at \url{https://youtu.be/KZ9GfT0J-e4}}
	\label{oneactive}
\end{figure*}

\begin{figure*}
		\setlength{\belowcaptionskip}{-14pt}
	\centering     
	\subfigure[$t=0.02s$]{\label{fig:a2}\includegraphics[width=0.24\textwidth,]{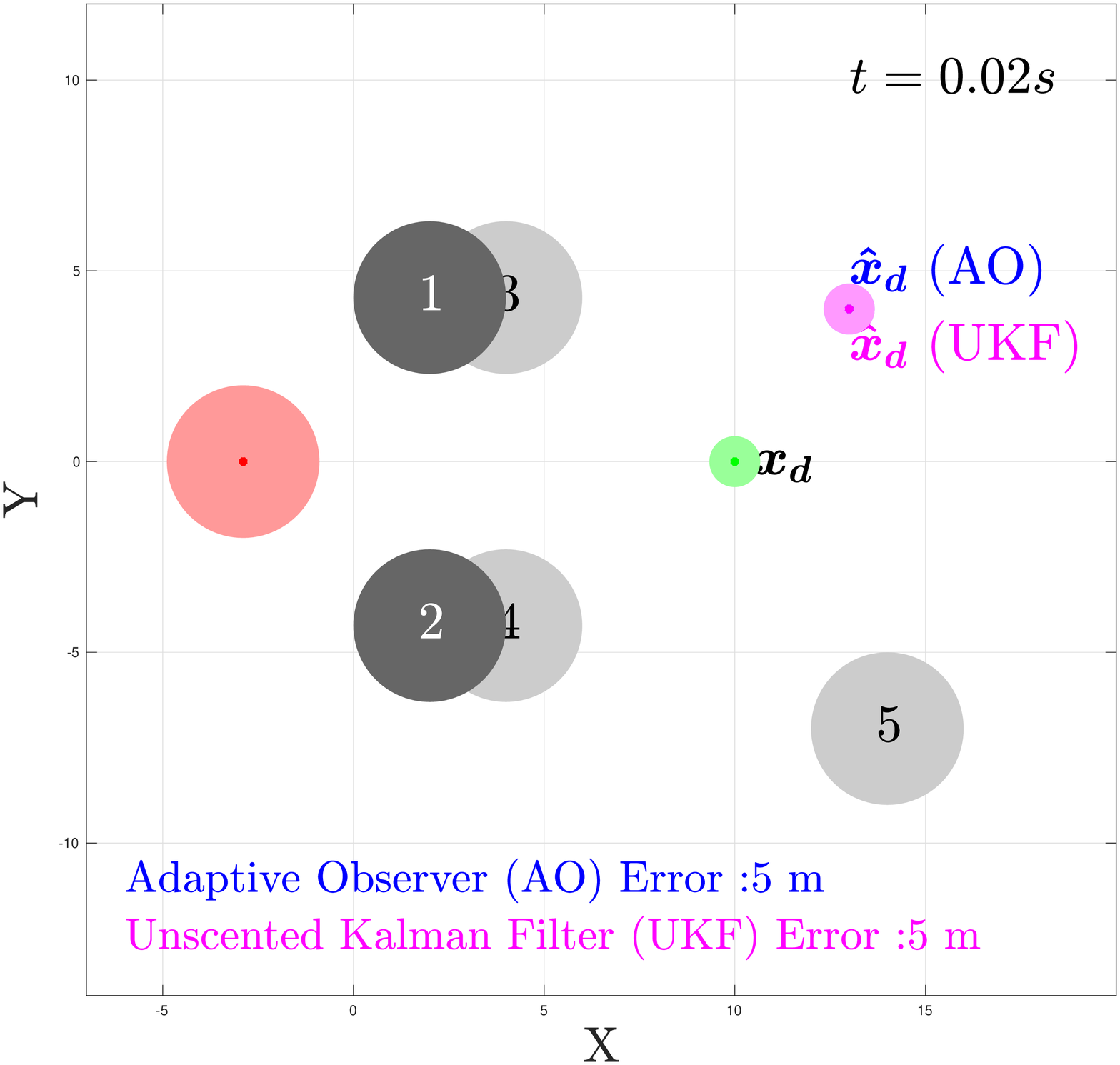}}
	\subfigure[$t=1.5s$]{\label{fig:b2}\includegraphics[width=0.24\textwidth]{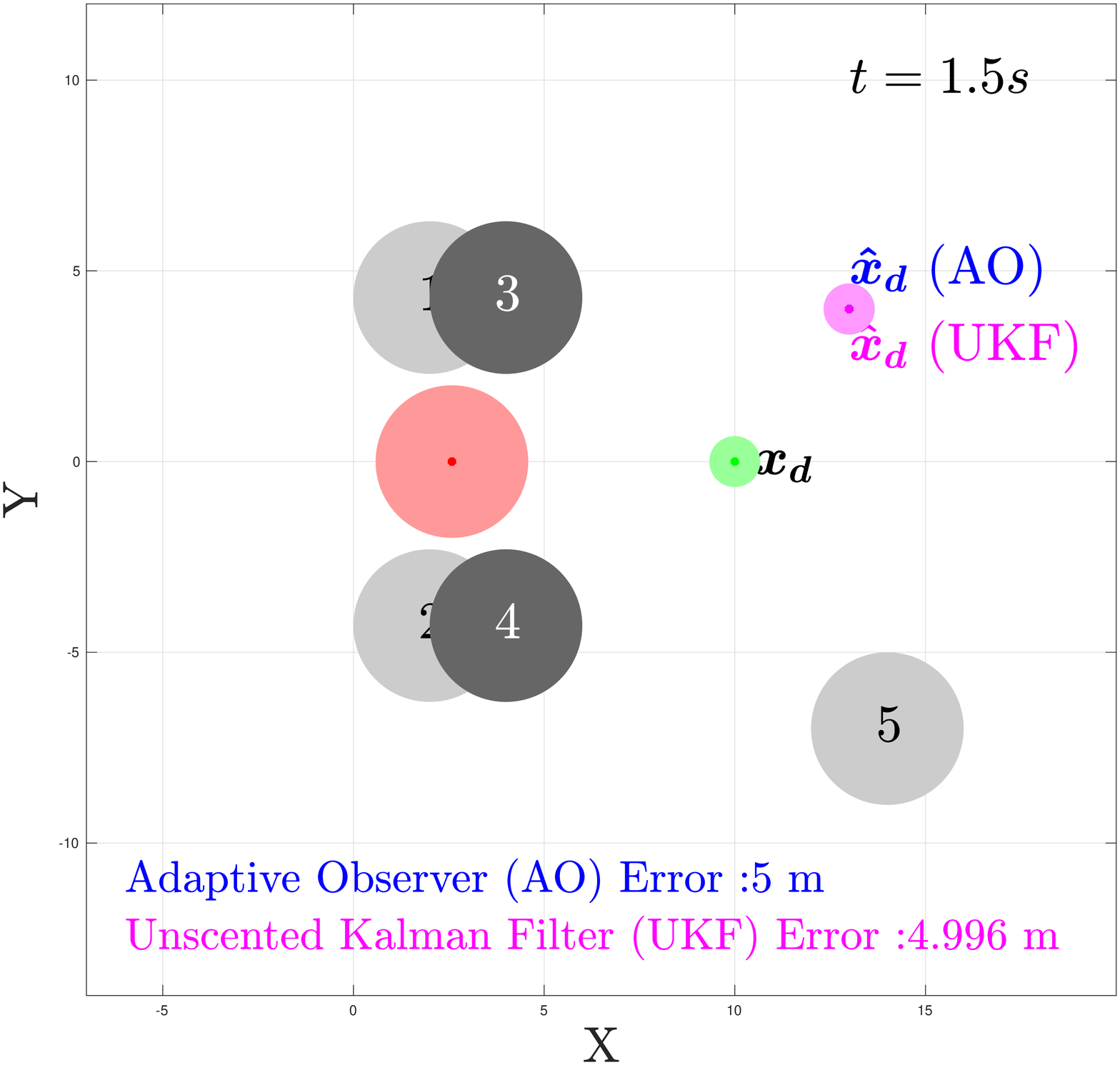}}
	\subfigure[$t=3s$]{\label{fig:c2}\includegraphics[width=0.24\textwidth]{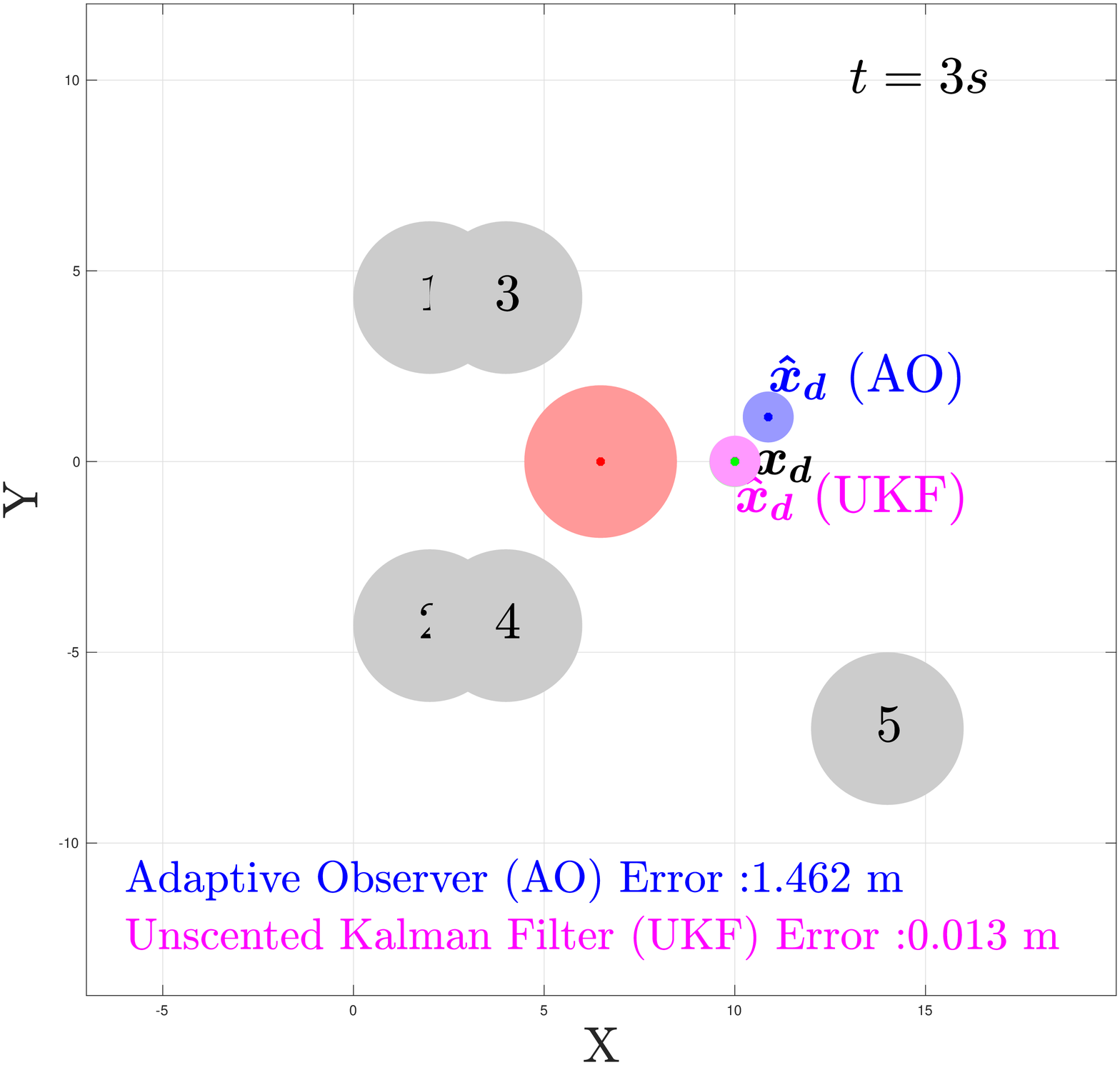}}
	\subfigure[Adaptive Observer and UKF]{\label{fig:d2}\includegraphics[width=0.24\textwidth]{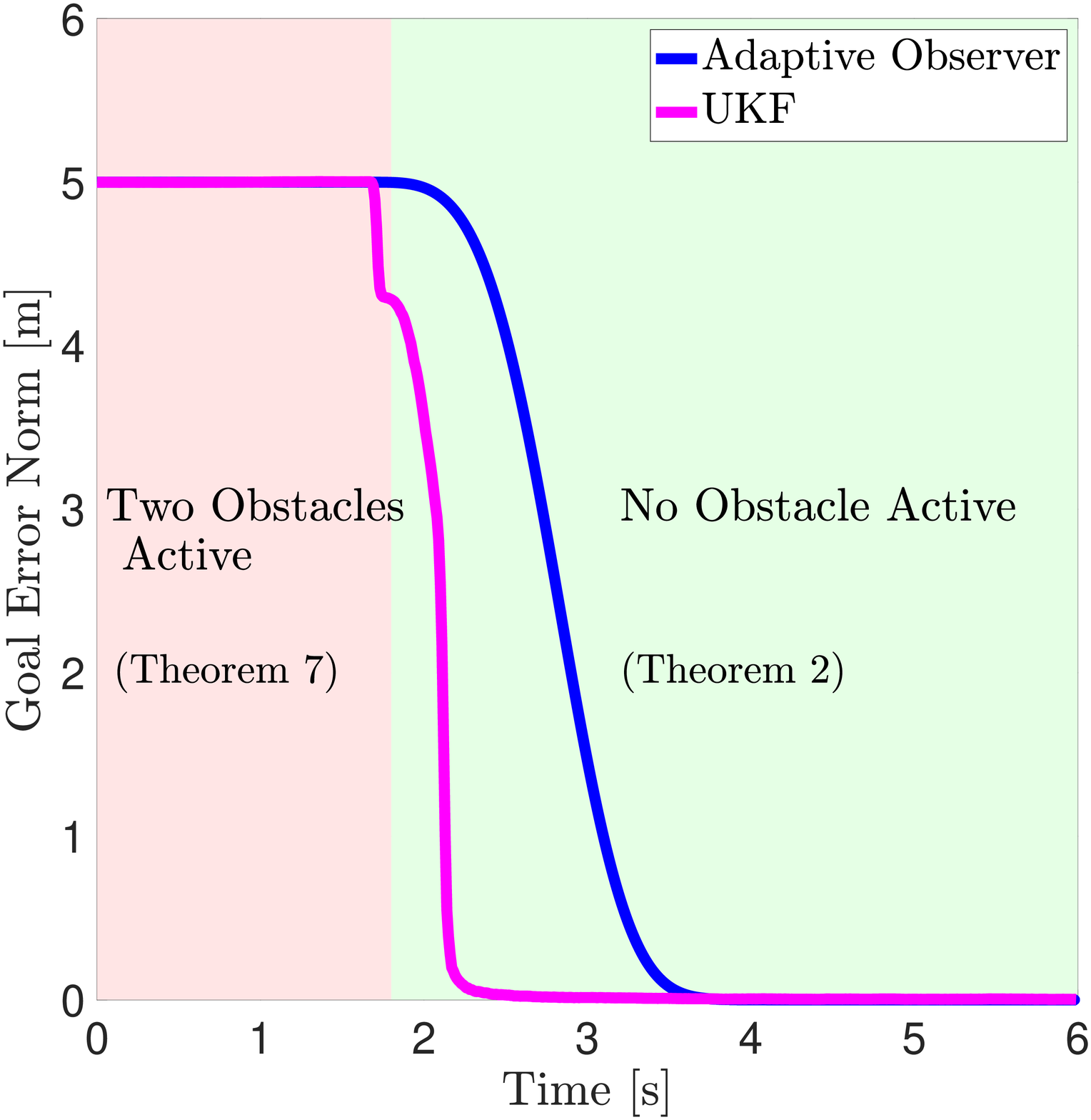}}
	\caption{(a)-(c) Goal identification for a robot navigating amongst static obstacles. (d) The red patch represents the duration in which identification is not supposed to work, because two obstacles are active.  After $t\sim1.8s$, no obstacles are active, hence estimation errors begin to converge to zero. Video at \url{https://youtu.be/baZNP0qZQrY}}
	\label{twoactive}
\end{figure*}
\subsection{$2\leq K \leq M $ and $\texttt{rank}(A_{ac}(\boldsymbol{x}))=2$}
\label{caseD}
Now consider the case where  \textit{two} of $K$ constraints are linearly independent, while the remaining $K-2$ constraints are linear combinations of these two. In this case,  there are fewer degrees of freedom in control than the number of independent active obstacles to avoid, hence $ \boldsymbol{u}^*_{\boldsymbol{\theta}}$ is completely determined by these constraints and does not depend on $\hat{\boldsymbol{u}}$, and by extension, neither on $k_p,\boldsymbol{x}_d$.  We formally demonstrate this claim as follows.  Let $i_1,i_2,\cdots,i_K \in \{1,2,\cdots,M\}$ be the indices of the $K$ constraints that are active. These constraints satisfy \eqref{abK} except that here $\texttt{rank}(A_{ac}(\boldsymbol{x}))=2$.  This problem is overdetermined but not ill-posed because we know that by construction \eqref{abK} is consistent. Its solution is given by solving the least-squares problem
\begin{align}
\label{u_determine}
\begin{aligned}
\boldsymbol{u}^*_{\boldsymbol{\theta}}(\boldsymbol{x}) &= \underset{\boldsymbol{u}}{\arg\min}
\norm{A_{ac}(\boldsymbol{x})\boldsymbol{u}-\boldsymbol{b}_{ac}(\boldsymbol{x})}_2^2 \\
&=A^{\dagger}_{ac}(\boldsymbol{x})\boldsymbol{b}_{ac}(\boldsymbol{x}),
\end{aligned}
\end{align}
where $A^{\dagger}_{ac}(\boldsymbol{x})$ denotes the Moore-Penrose pseudo inverse defined as $A^{\dagger}_{ac}(\boldsymbol{x}) \coloneqq (A^{T}_{ac}(\boldsymbol{x})A_{ac}(\boldsymbol{x}))^{-1}A^{T}_{ac}(\boldsymbol{x})$. When $K=2$,  $A^{\dagger}_{ac}(\boldsymbol{x}) \equiv A^{-1}_{ac}(\boldsymbol{x})$. Since neither $A^{\dagger}_{ac}(\boldsymbol{x})$ nor $\boldsymbol{b}_{ac}(\boldsymbol{x})$ depend on $k_p,\boldsymbol{x}_d$ \eqref{Ab_static}, parameter inference is \textbf{not} possible. 
\begin{theorem}
	\label{theorem_caseD}
	If $\forall t \in [0,T]$, two or more than two constraints are active and two of these constraints are linearly independent, then the observer \textbf{cannot} estimate either the goal or the gain, using $\boldsymbol{x}(t),\boldsymbol{u}^*(\boldsymbol{x}(t)) \mbox{ } \forall t \in [0,T]$. 
\end{theorem}

\begin{proof}
	 \begin{enumerate}[label=(\alph*)]
	     \item For $\boldsymbol{\theta}=\boldsymbol{x}_d$, define $G(\boldsymbol{x})\coloneqq 0^{2 \times 2}$, $f(\boldsymbol{x})\coloneqq A^\dagger_{ac}(\boldsymbol{x})\boldsymbol{b}(\boldsymbol{x})$ using \eqref{u_determine}. Therefore, the \textbf{IE} condition (Def. \ref{IE} and \eqref{GisPE}) is never satisfied.
	      \item For $\boldsymbol{\theta}=k_p$, define $G(\boldsymbol{x})\coloneqq \boldsymbol{0}$, $f(\boldsymbol{x})\coloneqq A^\dagger_{ac}(\boldsymbol{x})\boldsymbol{b}(\boldsymbol{x})$. The \textbf{IE} condition (Def. \ref{IE} and \eqref{GisPE}) is never satisfied.
	       \vspace{-0.40cm}
	 \end{enumerate} 
\end{proof}

\section{Simulation Results}
\label{Results}

In this section, we present results for estimation of goals for a multirobot system. We consider two online estimators, (a) an Unscented Kalman Filter (UKF) as a baseline and (b) an adaptive observer (AO) in section \ref{AdaptiveObserverAlgorithm} based on equations  \eqref{original} to \eqref{parameter_update_law}. $G(\boldsymbol{x}),f(\boldsymbol{x})$ for AO are chosen per theorems \ref{theorem_caseA}-\ref{theorem_caseD} by checking how many obstacles are active at a given time. The AO converges only when the necessary conditions in these theorems are satisfied. As for UKF, it doesn't require such explicit dynamics but then there are also no guarantees for convergence.  To substantiate this,  we first show simulations for a single robot navigating towards its goal in an environment consisting of static obstacles. In Figs. \ref{fig:a1}-\ref{fig:c1}, an ego robot (red) is trying to reach its goal shown in green. As the robot moves to the right, obstacle two remains active until $t=2.8s$. While there are six obstacles, only one of them is active in this duration, hence theorem \ref{theorem_caseB} guarantees reduction in goal estimation error using AO.  This is shown in the dark green left panel of Fig. \ref{fig:d1}. For $t>2.8s$, no obstacle is active hence theorem \ref{theorem_caseA} ensures that the goal estimation error converges to zero using AO as evident in Fig \ref{fig:d1}. A similar trend is obtained using UKF.

In Figs. \ref{fig:a2}-\ref{fig:c2}, the same robot is trying to reach its goal in green. We have purposefully positioned the obstacles in such a way that as the robot moves, obstacle one and two are active until $t=1.08s$, at which point, obstacle three and four become active, and stay so until $t=1.8s$. Thus until $t=1.8s$, two obstacles are always active. Hence, from theorem \ref{theorem_caseD}, robot dynamics do not depend on the goal location.  As expected, the AO does not update its estimate as can be seen in the red panel of Fig. \ref{fig:d2}, where goal error does not decrease. \textit{Interestingly, this is also true for UKF, which empirically speaks to the fact that convergence of estimation error is agnostic to the choice of estimator}, which is because robot dynamics itself do not depend on the goal for $t<1.8s$ However, this claim requires formal analysis. Thereafter, no obstacle is active, hence the estimation errors converge to zero as is evident from the green panel in Fig. \ref{fig:d2}.

Finally, we consider a multirobot system in Fig. \ref{MultirobotSnapshots} in which we run parallel estimators synchronously. To ensure that the snapshots are legible, we only highlight the ego robot (\textit{i.e.} robot 2), while other robots are light grey or dark depending on whether they are active or inactive for the ego robot. In this simulation, there are times when one robot is active (Fig. \ref{fig:a3}), two are active (Fig. \ref{fig:b3}) and none are active (Fig. \ref{fig:c3}). The estimation errors are shown in Fig.  \ref{fig:MultirobotGraphComparison}. The grey curves correspond to the non-ego robots and the blue (AO) and pink (UKF) curves are for the ego robot. Since all the curves converge to zero, the estimates of goals for all robots converge to their true goals.
\begin{figure*}
	\setlength{\belowcaptionskip}{-15pt}
	\centering     
	\subfigure[$t=0.02s$]{\label{fig:a3}\includegraphics[width=0.24\textwidth]{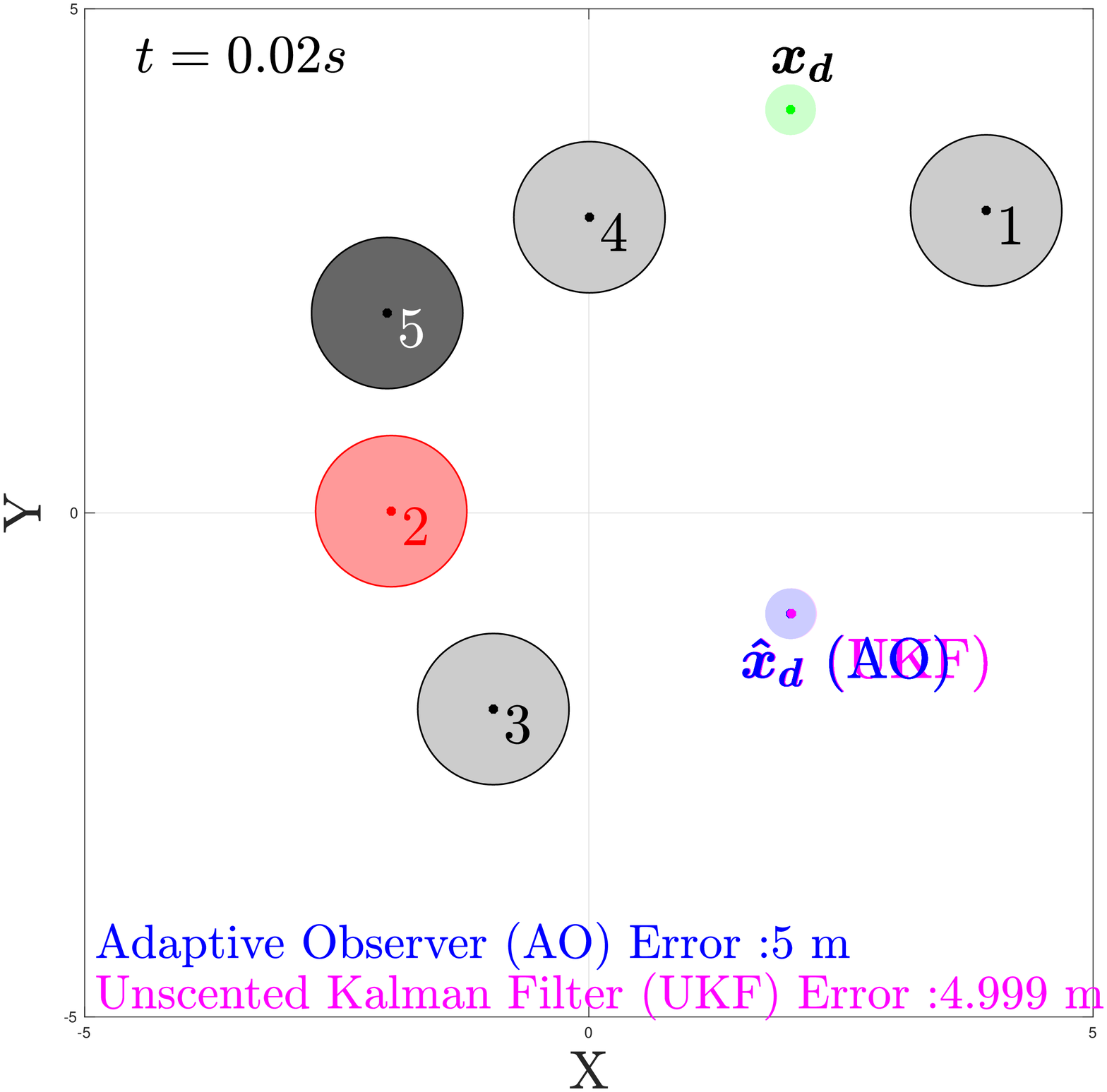}}
	\subfigure[$t=2s$]{\label{fig:b3}\includegraphics[width=0.24\textwidth]{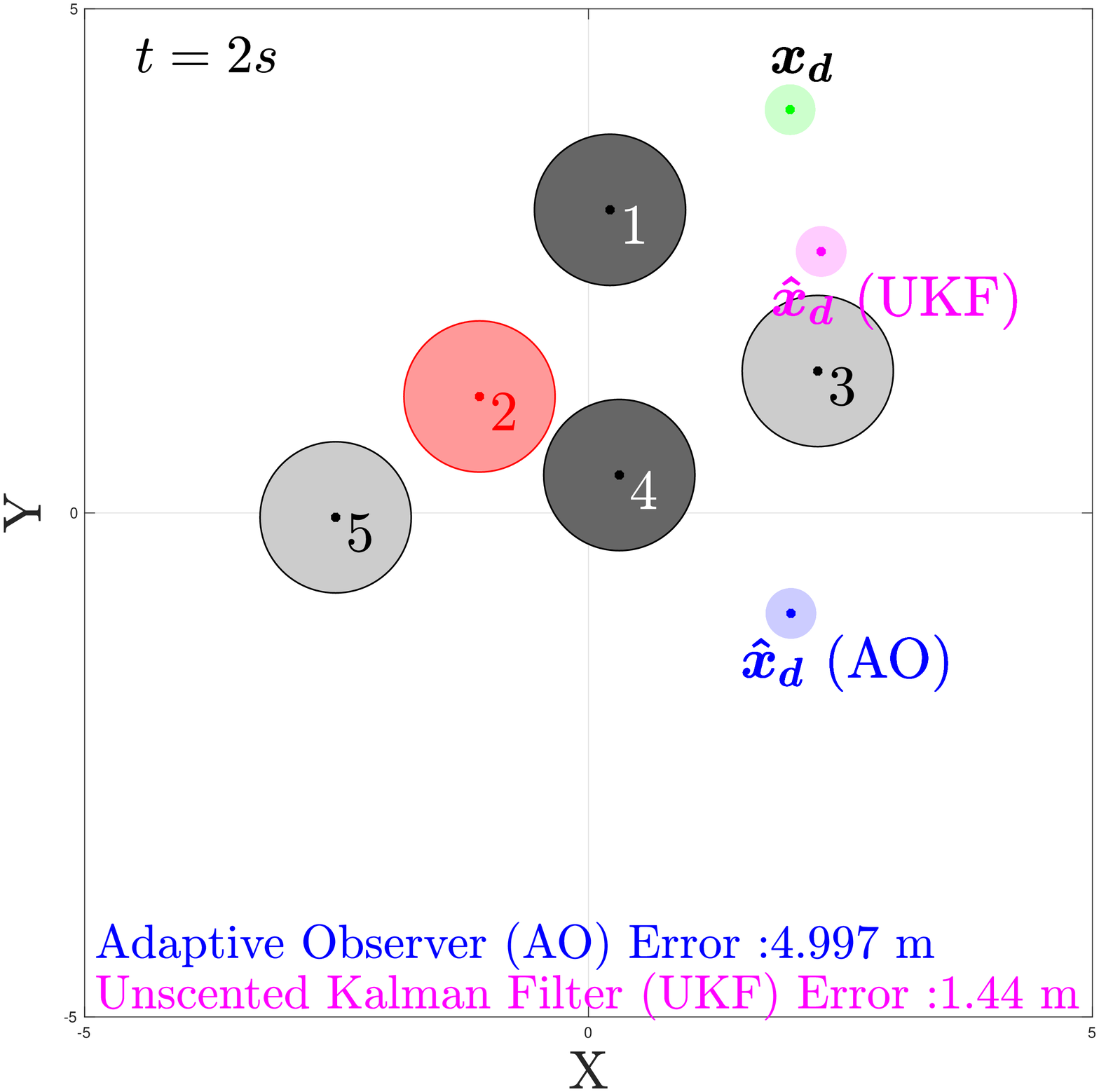}}
	\subfigure[$t=4s$]{\label{fig:c3}\includegraphics[width=0.24\textwidth]{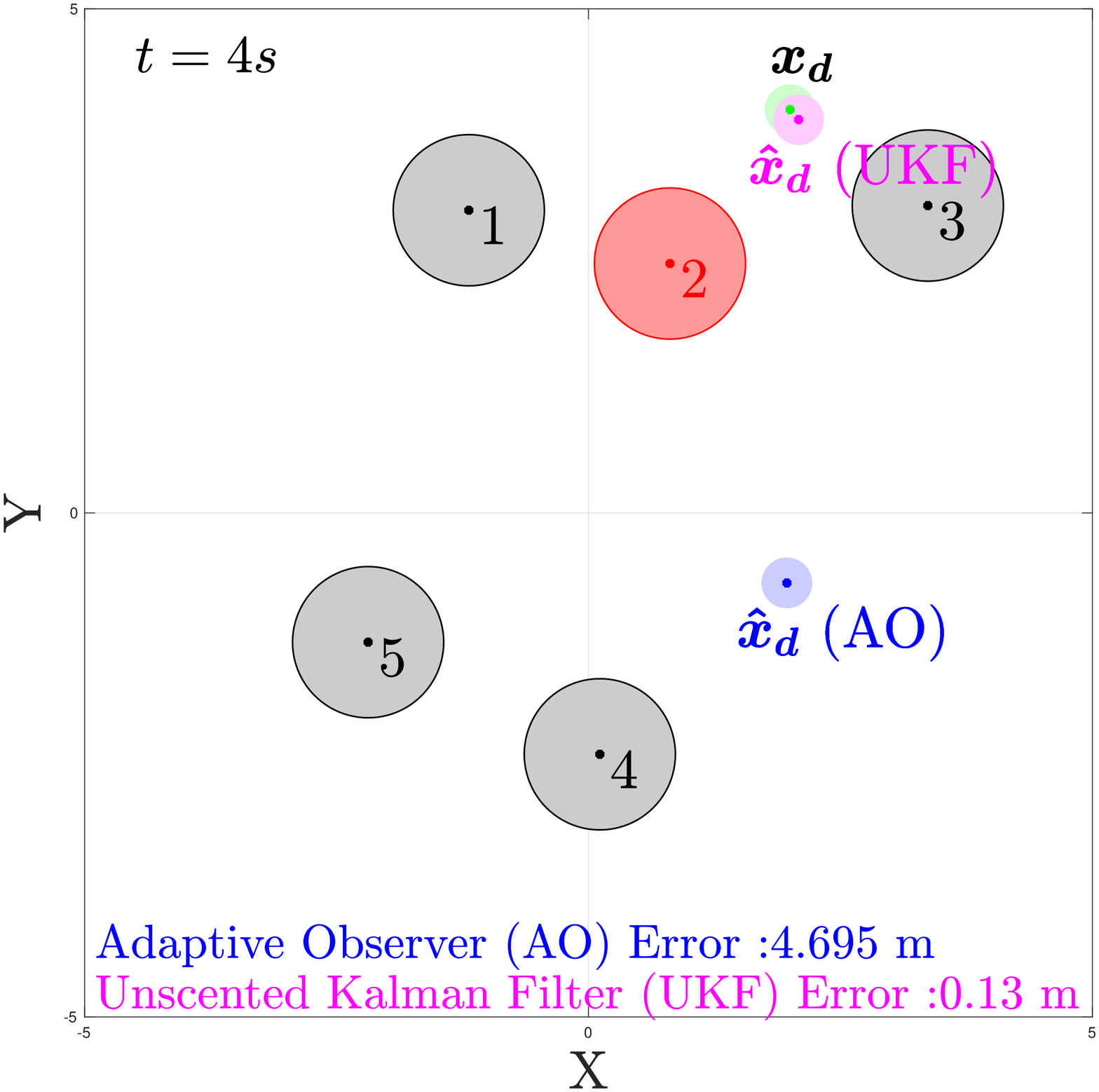}}
	\subfigure[$t=6s$]{\label{fig:d3}\includegraphics[width=0.24\textwidth]{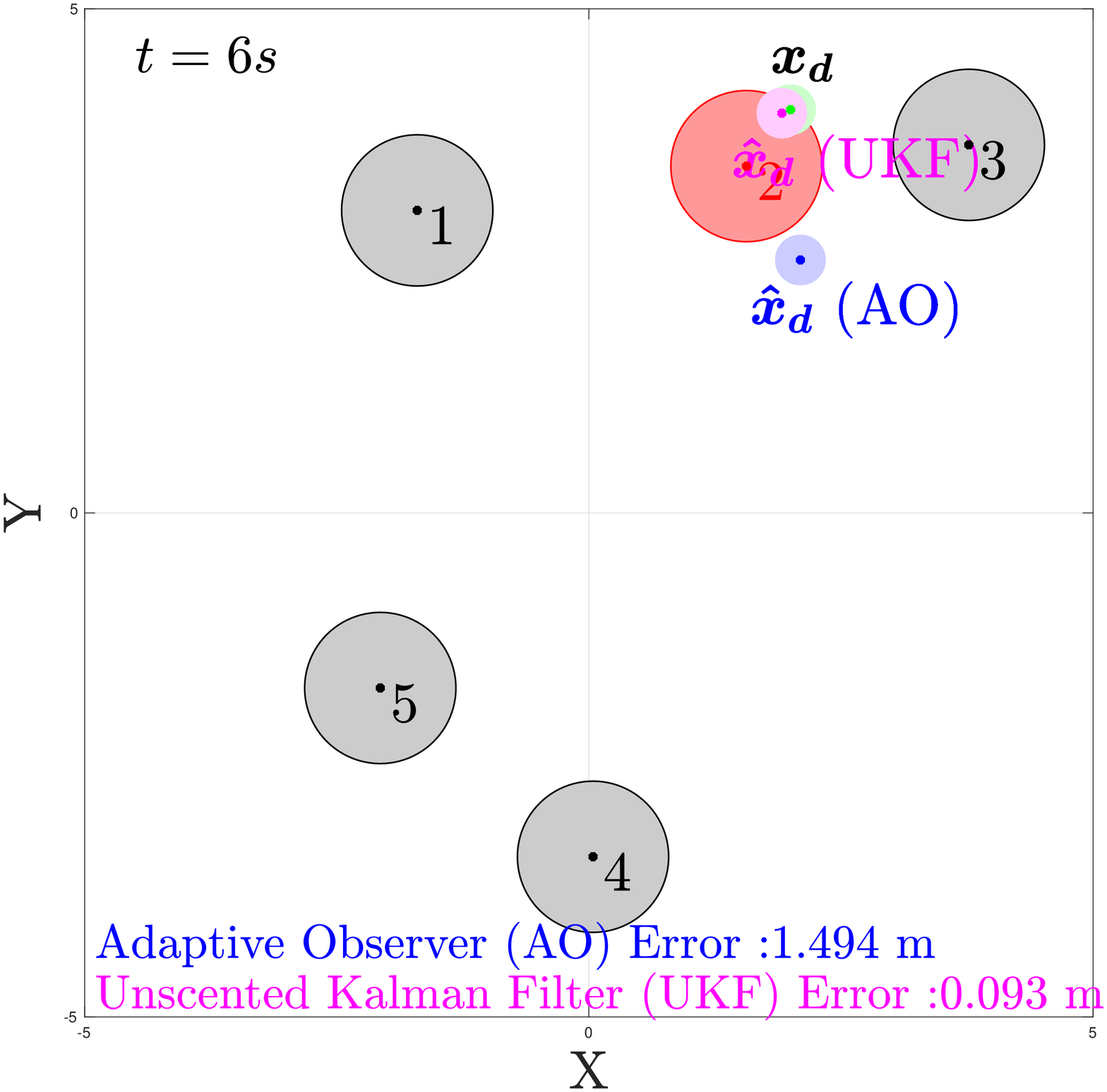}}
	\caption{(a)-(c) Goal estimation for a multirobot system. We highlight the ego robot (red) for legibility. AO and UKF estimates are shown in blue and pink discs respectively. Video at \url{https://youtu.be/-8qpBL5v4WY} }
	\label{MultirobotSnapshots}
\end{figure*}
\begin{figure}
	\centering
	\setlength{\belowcaptionskip}{-15pt}
	\includegraphics[width=\columnwidth]{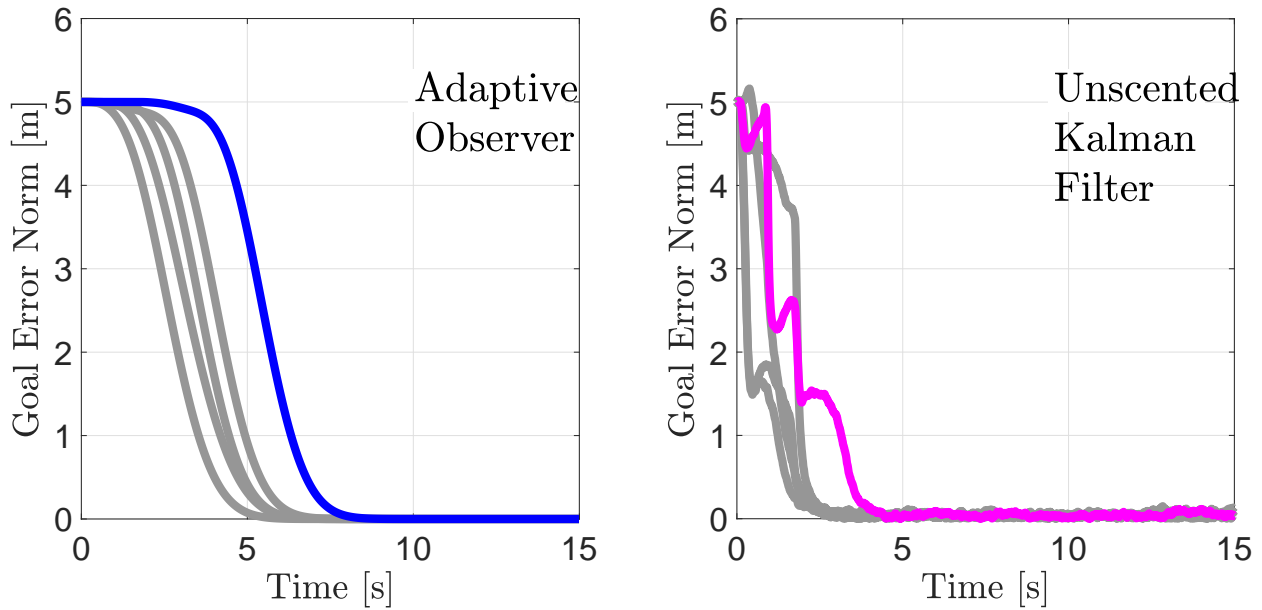}
	\caption{Estimation errors as a function of time for AO (left) and UKF (right). The highlighted curves are for the ego robot while the gray ones are for other robots.}
	\label{fig:MultirobotGraphComparison}
\end{figure}
\section{Conclusions}
\label{Conclusions}
In this paper, we developed the mathematical framework for observer based idenfication of goals and gains for a multirobot system by borrowing ideas from system identification. Since these robots use optimization in the feedback loop, their dynamics do not explicitly depend on parameters which makes the application of previously developed identifiability conditions non-trivial. We used duality theory to derive explicit relations between parameters and dynamics to derive identifiability conditions. The message that our theorems convey is that as the number of robots that an ego robot interacts with increases, estimation of ego's parameters becomes difficult because with more interactions, the ego robot's motion is expended in avoiding collisions which it achieves by sacrificing task performance. Our theory is fairly general, we intend to demonstrate this for inference of other multirobot tasks in future work. We also intend to apply this for estimation of parameters of a human performing a task and deduce what interventions from an external robot can make identification of human's task easier. This will have widespread applications in human robot interaction.
\bibliographystyle{ieeetr}
\bibliography{cmu}

\end{document}